\documentclass[11pt,a4paper]{article}
\usepackage[utf8]{inputenc}
\usepackage{amsfonts}
\usepackage{amsmath}
\usepackage{amssymb}
\usepackage{amsthm}
\usepackage{amscd}
\usepackage[g]{esvect}
\usepackage[mathscr]{eucal}
\usepackage[all]{xy}
\usepackage{textcomp,multicol,enumitem}
\usepackage{array}

\usepackage[Sonny]{fncychap}
\usepackage{footnote}
\usepackage{geometry}
\usepackage{appendix}
\usepackage[pdftex,pdfborder={0 0 0},linktoc=all]{hyperref}
\usepackage[nohints]{minitoc}

\usepackage{pgf,tikz}
\usetikzlibrary{arrows}

\usepackage{xcolor}
\usepackage{multicol}
\usepackage{pdfpages}

\newtheorem{theo}{Theorem}
\newtheorem{lemma}{Lemma}
\newtheorem{prop}{Proposition}

\newtheorem{rem}{Remark}

\numberwithin{equation}{section}
\numberwithin{figure}{section}

\newcommand{\N}{\mathbb{N}}
\newcommand{\R}{\mathbb{R}}

\renewcommand{\geq}{\geqslant}
\renewcommand{\leq}{\leqslant}
\renewcommand{\tilde}{\widetilde}

\title{Stability for small data: the drift model of the conformal method}
\author{Caterina Vâlcu\\ \'Ecole polytechnique\thanks{CMLS, \'Ecole Polytechnique, F-91128 Palaiseau Cedex, France.  E-mail: \texttt{maria-caterina.valcu@polytechnique.edu}}}
\date{\today}

\begin{document}
\newgeometry{left=3.5cm,right=2.5cm,top=2.5cm,bottom=3cm}

\maketitle
\begin{abstract} The conformal method in general relativity aims to successfully parametrise the set of all initial data associated with globally hyperbolic spacetimes. One such mapping was suggested by David Maxwell \cite{Max14b}. I verify that the solutions of the corresponding conformal system are stable, in the sense that they present \textit{a priori} bounds under perturbations of the system's coefficients. This result holds in dimensions $3\leq n\leq 5$, when the metric is conformally flat, the drift is small. A scalar field with suitably high potential is considered in this case.
\end{abstract}
\section{Introduction}
A spacetime is defined as the equivalence class, up to an isometry, of Lorentzian manifolds $(\tilde{M},\tilde{g})$ of dimension $n+1$, which satisfy the Einstein field equations
\begin{equation}
Ric_{\alpha\beta}(\tilde{g})-\frac{1}{2}R(\tilde{g})\tilde{g}_{\alpha\beta}=8\pi T_{\alpha\beta},
\end{equation}
$\alpha,\beta=\overline{1,n+1}$. We have used the following notation: $R(\tilde{g})$ is the scalar curvature of $\tilde{g}$, $Ric$ the Ricci curvature and $T_{\alpha\beta}$ the stress-energy tensor. If $T_{\alpha\beta}=0$, we describe the vacuum. If
\begin{equation}
T_{\alpha\beta}=\tilde{\nabla}_\alpha\tilde\psi\tilde{\nabla}_\beta\tilde\psi-\left(\frac{1}{2}|\tilde{\nabla}\tilde\psi|^2_{\tilde{g}}+V(\tilde\psi)\right)\tilde{g}_{\alpha\beta},
\end{equation}
the model corresponds to the existence of a scalar field $\tilde\psi\in\mathcal{C}^\infty(M)$ having potential $V\in\mathcal{C}^\infty(\R)$. By correctly choosing $\psi$ and $V$, we can describe the vacuum with cosmological constant and the Einstein-Klein-Gordon setting.
\par A globally hyperbolic spacetime accepts initial data $(M,g,K,\psi,\pi),$ where 
\begin{itemize}
\item $(M,g)$ is an n-dimensional Riemannian manifold,
\item $K$ is a symmetric 2-tensor corresponding to the second fundamental form,
\item  $\psi$ represents the scalar field in $M$, and
\item $\pi$ is its derivative.
\end{itemize}
The associated spacetime development takes the form $(M\times \R,\tilde{g}, \tilde{\psi})$, where $\tilde{g}$ is a Lorentzian metric that verifies $\tilde{g}|_{M}=\hat{g}$ and $\tilde\psi$ is a scalar field such that $\tilde{\psi}|_{M}=\hat{\psi}$ and $\partial_t\tilde{\psi}|_{M}=\hat{\pi}$.
\par Through the work of Choquet-Bruhat and Geroch, having the initial data verify the constraint equations is proved to be not only a necessary, but a sufficient condition for the development of a maximal, globally hyperbolic space-time \cite{Cho52,ChoGer69}:
\begin{equation}
\begin{array}{r l}
 R(\hat{g})+(tr_{\hat{g}}\hat{K})^2-|\hat{K}|^2_{\hat{g}} &= \hat{\pi}^2+|\hat{\nabla}\hat{\psi}|^2_{\hat{g}}+2V(\hat{\psi})\\
\partial_i(tr_{\hat{g}} \hat{K})-\hat{K}_{i,j}^j&= \hat{\pi}\partial_i\hat{\psi},
\end{array}
\end{equation}
The above system is clearly under-determined, which allows for a good amount of freedom in choosing $(\hat{g},\hat{K},\hat{\psi},\hat{\pi})$.
\par The conformal method began with Lichnerowicz \cite{Lic44}, and was later developped by York, Jr., \'{O} Murchadha and Pfeiffer,  \cite{Yor73,OMuYor74,Yor99,PfeYor03}. We allow for
the constraint equations to be transformed into a determined system of equations by fixing well-chosen quantities (see Choquet-Bruhat, Isenberg and Pollack \cite{ChoIsePol07}). Essentially, the technique maps a space of parameters to the space of solutions.
\par Given an initial data set $(\hat{g},\hat{K},\hat{\psi},\hat{\pi})$, the classical choice of parameters is $(\mathbf{g},\mathbf{U},\tau,\psi,\pi;\alpha)$: in this case, the conformal class $\mathbf{g}$ is represented by a Riemannian metric $g$, the smooth function  $\tau=\hat{g}^{ab}\hat{K}_{ab}$ is a mean curvature and the conformal momentum $\mathbf{U}$ measured by a volume form $\alpha$ (volume gauge) is a 2-tensor that is both trace-free and divergence-free with respect to $g$ (a transverse-traceless tensor). We sometimes prefer to indicate the volume gauge by the densitized lapse 
\begin{equation}
\tilde{N}_{g,\alpha}:=\frac{\alpha}{dV_g}.
\end{equation}
In 2014, Maxwell introduces a variant to the standard conformal method called ``the drift method" \cite{Max14b}. Very succinctly, it differs from its predecessor in that it replaces the mean curvature $\tau$ with two new conformal data, a volumetric momentum and a drift. These new quantities are defined by the volumetric equivalent to the York splitting \cite{Max15}:
\begin{equation}
\tau=\tau^*+N_{g,\omega}div_{g} (V+Q)
\end{equation}
where $\tau^*\in\R$, $V$ is a smooth vector field and $Q$ is a conformal Killing field.
\par Given a gauge $\omega$, one might choose
\begin{enumerate}
\item an arbitrary representative $g_{ab}\in\mathbf{g}$,
\item the unique densitized lapse $\tilde{N}_{g,\omega}$,
\item the unique TT-tensor $U_{ab}$ such that $(g_{ab},U_{ab})=\mathbf{U}$, where $\mathbf{U}$ is the conformal momentum as measured by $\omega$, and
\item a vector field $\tilde{V}$, unique up to a conformal Killing field, such that $(g_{ab},\tilde{V}_{ab})=\mathbf{V}$, where $\mathbf{V}$ is the volumetric drift measured by $\omega$.
\end{enumerate}
We denote by
\begin{equation}
q:=\frac{2n}{n-2}
\end{equation}
the critical Sobolev constant corresponding to the embedding of $H^1$ into the Lebesgue spaces. Let
\begin{equation}
\mathcal{L}_g W_{ij}=W_{i,j}+W_{j,i}-\frac{2}{n}div_g W g_{ij}
\end{equation}
be the conformal Killing operator with respect to $g$. We use the decompositions
\begin{equation}
\begin{array}{r l}
 \hat{g}_{ab}&=u^{q-2}g_{ab}  \\
 \hat{K}_{ab}&=u^{-2}[\frac{\tilde{N}}{2}(\mathcal{L}_gW)_{ab}+U_{ab}]+\frac{1}{n}u^{q-2}g_{ab}\left(\tau^*+\tilde{N}div_g(V+Q)\right).
\end{array}
\end{equation}
where $u$ is a scalar function and $W$ and $Q$ are vector fields, all unknown.
\par For additional details on the drift method, see the annex: section \ref{drift method continued}.
\subsection{The viability of conformal method models}\label{viability section}
The conformal method essentially provides a mapping from the set of conformal data representatives to the set of initial data,
\begin{equation}
\begin{array}{r c l}
\text{Conformal data representatives}&\rightarrow& \text{Initial data.}
\end{array}
\end{equation}
More precisely, in the case of the classical conformal method, given a volume gauge $\omega$, the mapping presents as
\begin{equation}
\begin{array}{r c l}
(g_{ab},U_{ab},\tau; N)&\xrightarrow{\text{solve} (u,W)}&(\hat{g},\hat{K}).
\end{array}
\end{equation}
By the nature of the conformal method, the mapping is unto: from any set of initial data, one can calculate a set of corresponding conformal data representatives. We list a number of criteria by which the strength of a conformal method may be judged.
\begin{description}
\item[1)] Is the mapping a bijection?
\end{description}
Ideally, to any set of conformal data representatives there corresponds 
one and only one set of initial data. Thus, the set of all possible initial data is completely characterized by the conformal method. 
\par In lieu of such a strong result, one may ask:
\begin{itemize}
\item Where is the mapping well-defined (in the sense that there exists $(\hat{g},\hat{K})$ corresponding to a fixed set of conformal data representatives)? As long as we properly identify the problem sets, we can simply remove them from the domain.
\item Where is the mapping one-to-one? If we obtain multiple solutions, where does this happen? 
\end{itemize}  
\begin{description}
\item[2)] Is the mapping continuous?
\end{description}
This question tests that the mapping is, in some sense, physically relevant.
\subsection{The main result.}
Let $(M,g)$ be a closed locally conformally flat manifold of dimension $n$, which can be $3$, $4$ and $5$.  Let
\begin{equation}
\Delta_g= -div_g\nabla
\end{equation}
be the Laplace-Beltrami operator with non-negative eigenvalues. Similarly, let 
\begin{equation}
\overrightarrow{\Delta}_g W_i=-div_g(\mathcal{L}_g W)_i=-(\mathcal{L}_g W)_{ij,}^{\,\,\,\,\,\, j}
\end{equation}
be the corresponding Lam\'e operator. The volumetric drift model proposed by Maxwell leads to the reworking of the Einstein constraint equations as
\begin{equation}\label{syst of Maxwell2}
\begin{array}{r l}
\Delta_g u+\frac{n-2}{4(n-1)}(R(g)+\vert\nabla\psi\vert_g^2)u= &\frac{(n-2)\vert U+\mathcal{L}_g W\vert^2+\pi^2}{4(n-1)u^{q+1}}\\ \\
&+\frac{n-2}{4(n-1)}\Big(2V(\psi)-\frac{n-1}{n}\tau^*\\ \\
&\quad\quad\quad\quad\quad +\frac{n-1}{n}\frac{\tilde{N}div_g(u^q\tilde{V})^2}{u^{2q}}\Big)u^{q-1} \\ \\
div_g\left(\frac{\tilde{N}}{2}\mathcal{L}_g W\right)=&\frac{n-1}{n}u^q \mathbf{d}	\left(\frac{\tilde{N}div_g(u^{q}\tilde{V})}{2u^{2q}}\right)+\pi\nabla\psi.
\end{array}
\end{equation}
The existence of solutions to this system was treated in \cite{HolMaxMaz18} in the non-focusing case, and in \cite{Val19} for the focusing case. The classical conformal method, also in the focusing regime, is treated in \cite{Pre14}. See \cite{DruHeb09} for the precursor of the asymptotic techniques used in the existence proofs.
\par The second equation may be rewritten as:
\begin{equation}
\begin{array}{r l}
 \overrightarrow{\Delta}_gW =&\langle \nabla\ln \tilde{N},\mathcal{L}_g W\rangle+2\frac{n-1}{n-2}\left(\frac{3n-2}{n-2}\frac{\left\langle\nabla u,\tilde{V}\right\rangle\nabla u}{u^2}-\frac{\left\langle\nabla^2 u,\tilde{V}\right\rangle}{u}\right)  \\
 &+2\frac{n-1}{n-2}\left(-\frac{\left\langle\nabla u,\tilde{V}\right\rangle}{u}\nabla\ln\tilde{N}+div_g\tilde{V}\frac{\nabla u}{u}-\frac{\langle \nabla \tilde{V},\nabla u\rangle}{u}\right)\\
 &-\frac{n-1}{n}\left(div_g\tilde{V}\nabla\ln\tilde{N}+\nabla div_g\tilde{V}\right)-2\tilde{N}^{-1}\pi\nabla\psi.
\end{array}
\end{equation}
In the present paper, it sometimes proves useful to work with the more general equation
\begin{equation}\label{systemshort2}
\begin{array}{r l}
\Delta_g u+hu = &  fu^{q-1}+\frac{\rho_1+\vert\Psi+\rho_2\mathcal{L}_gW\vert^2_g}{u^{q+1}} \\
 &\quad\quad -\frac{b}{u}-c\langle\nabla u,Y\rangle \left(\frac{d}{u^2}+\frac{1}{u^{q+2}}\right)-\frac{\langle\nabla u,Y\rangle^2}{u^{q+3}}\\
\end{array}
\end{equation}
where we make the following substitutions:
\begin{equation}\label{notation maxwell}
\begin{array}{l}
h=\frac{n-2}{4(n-1)}\left(\mathcal{R}_g-|\nabla\psi|^2_g\right),\quad f=\frac{n-2}{4(n-1)}\left(2V(\psi)-\frac{n-1}{n}\tau^{*2}\right),\\
\rho_1=\frac{n-2}{4(n-1)}\left(\pi-\frac{n-1}{n}\tilde{N}^2div_g\tilde{V}\right),\quad \rho_2=\sqrt{\frac{n-2}{n-1}}\frac{\tilde{N}}{4}, \quad \Psi=\sqrt{\frac{n-2}{n-1}}\frac{U}{2},\\
b=\frac{n-2}{2n}\tau^*\tilde{N}div_g\tilde{V},\quad c=\sqrt{\frac{n-2}{n}},\quad d=\tau^*,\quad Y=\sqrt{\frac{n}{n-2}}\tilde{N}\tilde{V}.\\
\end{array}
\end{equation}
Consider $(u_\alpha, W_\alpha)_{\alpha\in\N}$ a sequence of smooth solutions of perturbations of the system \eqref{systemshort2},
\begin{equation}\label{perturbed system2}
\begin{array}{r l}
\Delta_{g} u_\alpha+h_\alpha u_\alpha = &  f_\alpha u_\alpha^{q-1}+\frac{\rho_{1,\alpha}+\vert\Psi_\alpha+\rho_{2,\alpha}\mathcal{L}_gW_\alpha\vert^2_g}{u_\alpha^{q+1}} \\
 &\quad\quad -\frac{b_\alpha}{u_\alpha}-c_\alpha\langle\nabla u_\alpha,Y_\alpha\rangle \left(\frac{d_\alpha}{u_\alpha^2}+\frac{1}{u_\alpha^{q+2}}\right)-\frac{\langle\nabla u_\alpha,Y_\alpha\rangle^2}{u_\alpha^{q+3}}\\
\overrightarrow{\Delta}_g W_\alpha =
    & \mathcal{R}_\alpha(u_\alpha,\nabla u_\alpha,\nabla^2 u_\alpha,\mathcal{L}_gW_\alpha).
\end{array}
\end{equation}
Here, we ask that the perturbed coefficients converge towards the initial ones in a sufficiently regular way, \textit{e.g} in $\mathcal{C}^{2,\eta}$ norm. The scalar solutions $u_\alpha$ are positive as long as $\rho_1$ which is positive. To see this, let $m_\alpha=\min_{x\in M}u_\alpha(x)=u_\alpha(x_\alpha)>0$ and let
\begin{equation}
a_\alpha=\rho_{1,\alpha}+\vert\Psi_\alpha+\rho_{2,\alpha}\mathcal{L}_gW_\alpha\vert^2_g.
\end{equation}
Since $\nabla u_\alpha(x_\alpha)=0$ and since $\Delta_g u_\alpha(x_\alpha)\leq 0$, we have 
$$h_\alpha(x_\alpha)m_\alpha-f_\alpha(x_\alpha)m_\alpha^{q-1}-\frac{a_\alpha(x_\alpha)}{m_\alpha^{q+1}}+\frac{b_\alpha(x_\alpha)}{m_\alpha}\geq 0.$$
Since $a_\alpha\to a$ in ${\mathcal C}^0\left(M\right)$ as $\alpha\to +\infty$ and $a>0$ in $M$, there exists $\varepsilon>0$ such that $m_\alpha\geq \varepsilon$, meaning that
\begin{equation}\label{varepsilon2}
u_\alpha\geq\varepsilon>0\quad \text{ for all } x\in M \text{ and all } \alpha.
\end{equation}
We would like to prove the \textit{a priori} estimate 
\begin{equation}
\vert\vert u_\alpha\vert\vert_{\mathcal{C}^{2,\eta}}+\vert\vert W_\alpha\vert\vert_{\mathcal{C}^{1,\eta}}\leq C.
\end{equation}
If this is true, then by standard elliptic theory there exists, up to a subsequence, a $\mathcal{C}^{2 ,\eta}$ limit of $(u_\alpha,W_\alpha)$ solving the limiting system \eqref{systemshort2}. In effect, since the system \eqref{systemshort2} is invariant by the addition of conformal Killing fields, it suffices to show that
\begin{equation}
\vert\vert u_\alpha\vert\vert_{L^\infty}+\vert\vert\nabla u_\alpha\vert\vert_{L^\infty}+\vert\vert \nabla^2 u_\alpha\vert\vert_{L^\infty}+\vert\vert \mathcal{L}_g W_\alpha\vert\vert_{L^\infty}\leq C.
\end{equation}
The proof follows by contradiction. We assume instead that there exists a sequence of solutions $(u_\alpha,W_\alpha)$ of the perturbed system such that
\begin{equation}\label{hyp to contradict2}
\vert\vert u_\alpha\vert\vert_{L^\infty}+\vert\vert\nabla u_\alpha\vert\vert_{L^\infty}+\vert\vert \nabla^2 u_\alpha\vert\vert_{L^\infty}+\vert\vert \mathcal{L}_g W_\alpha\vert\vert_{L^\infty}\to\infty\quad\text{ as }\quad\alpha\to\infty.
\end{equation}
The main theorem in this paper is the following.
\begin{theo}\label{stability syst2}
Let $(M,g)$ be a closed Riemannian manifold of dimension $n=3,4,5$, where $g$ is locally conformally flat. Let $\frac{1}{2}<\eta<1$ and $0<\alpha<1$. Let $a$, $b$, $c$, $d$, $f$, $h$, $\rho_1$, $\rho_2$, $\psi$, $\pi$, $\tilde{N}$ be smooth functions on $M$, let $\tilde{V}$ and $Y$ be smooth vector field on $M$. For any $0<\theta<T$, there exists $S_{\theta,T}$ and $\vartheta_{\theta,T}$ such that, given any parameters within
\begin{equation}
\begin{array}{c}
\mathcal{E}_{\theta,T}:=\Big\{(f,a,b,c,d,h,\rho_1, \rho_2, Y)\times(\tilde{N},\tilde{V},\psi,\pi),
\quad f\geq \theta,\quad a\geq \theta,\quad \tilde{N}\geq \theta,\quad\quad\quad\quad\quad\quad\quad\quad\quad\quad\quad\quad \\ \\
||f||_{\mathcal{C}^{1,\eta}}\leq T,\quad\quad\quad\quad\quad\quad\quad\quad\quad\quad\quad\quad\quad\quad\quad\quad\quad\quad\quad\quad\quad\quad\quad\quad\quad\quad\quad\quad\\
\quad ||a||_{\mathcal{C}^{1,\alpha}},||b||_{\mathcal{C}^{1,\alpha}}, ||c||_{\mathcal{C}^{1,\alpha}}, ||d||_{\mathcal{C}^{1,\alpha}},\quad\quad\quad\quad\quad\quad\quad\quad\quad\quad\quad\quad\quad\quad\quad\quad\quad\quad\\ ||\rho_1||_{\mathcal{C}^{1,\alpha}}, ||\rho_2||_{\mathcal{C}^{1,\alpha}}, ||h||_{\mathcal{C}^{1,\alpha}}, ||Y||_{\mathcal{C}^{1,\alpha}}\leq T,\quad\quad\quad\quad\quad\quad\quad\quad\quad\quad\quad\\ \\ 
\text{and }\quad||\tilde{N}||_{\mathcal{C}^{2,\alpha}}, ||\tilde{V}||_{\mathcal{C}^{2,\alpha}}\leq T
\quad\quad\quad\quad\Big\},
\end{array}
\end{equation}
with
\begin{equation}\label{extra hyp2}
\vert\vert Y\vert\vert_{\mathcal{C}^{1,\alpha}}, \vert\vert \tilde{V}\vert\vert_{\mathcal{C}^{2,\alpha}}\leq \vartheta_{\theta,T},
\end{equation}
then any smooth solution $(u,W)$ \eqref{syst of Maxwell2}, with $u>0$, satisfies 
\begin{equation}
||u||_{\mathcal{C}^{2,\alpha}}+\vert\vert \mathcal{L}_g W\vert\vert_{\mathcal{C}^{1,\alpha}}\leq S_{\theta,T}.
\end{equation}
\end{theo}
A few remarks are in order at this point. We have taken the decision to write the theorem using the physical coordinates for the second equation, and the general coefficients for the first. The same is true for the ensuing proof. This forcibly leads to some redundancies. We recall that $Y=\sqrt{\frac{n}{n-2}}\tilde{N}\tilde{V}$, so asking for bounds on $\tilde{N}$ and $\tilde{V}$ imply bounds on $Y$. The reasons why we still choose this writing are as follows:
\begin{enumerate}
\item The general notation of the first equation is the same as the ones used in the paper proving the existence of solutions to the system, and are more readable than the physical coordinates one. Moreover, they more accurately capture the nature of the scalar equation and make it easier to handle, since one can follow each of the different non-linear terms separately.
\item Writing the second equation in more general terms can prove counterproductive. For one, introducing new coefficients would actually burden the notation in this particular case. 
\item Most importantly, one hopes that there is a better way to treat potential blow-ups caused by $\mathcal{L}_g W_\alpha$. This could follow from a more detailed analysis of the second equation, where even the exact size of each of the dimensional constants can potentially play a role, given the coupling of the system.  
\end{enumerate}
\par In the proof, we use the smallness of $Y$ (and thus, $\tilde{V}$, since $\tilde{N}\geq \theta$) as sparsely as possible, and we take care to emphasize it each time. We do this out of the desire to provide what we hope is useful insight into current technical difficulties. By looking at similar systems, such as the  Yamabe problem, one can hope that by advancing the necessary techniques, one can successfully remove the smallness hypothesis altogether. The ultimate goal is to work towards a more proof of a more general stability.
\par For now, in the argument by contradiction, we are working with
\begin{equation}
Y_\alpha\to 0\quad\text{ in }\mathcal{C}^{1,\alpha}\quad\text{ and }\quad \tilde{V}_\alpha\to 0\quad\text{ in }\mathcal{C}^{2,\alpha}.
\end{equation}
\par The fact that $g$ is locally conformally flat is a condition we impose to get the improved estimates on  $\mathcal{L}_g W_\alpha$ that we need. We briefly explain the reasoning. The Green representation formula is applied on balls of diminishing radius $B_{x_\alpha}(\delta_\alpha)$, $\delta_\alpha\to 0$, where $x_\alpha$ is a concentration point. Moreover, we impose Neumann boundary conditions, so that there is no dependency on $W_\alpha$, but just $\mathcal{L}_{g_\alpha} W_\alpha$. The bounds needs to be uniform with respect to $\alpha$, which is why we need the kernel of $\overrightarrow{\Delta}_{g_\alpha}$ to have the same dimension as that of $\overrightarrow{\Delta}_\xi$, with $g_\alpha=\exp^*_{x_\alpha}(\delta_\alpha\cdot)$. 
\par The stability of the classical system, also in the focusing case, was treated in \cite{Pre16}.
\bigskip
\par\textbf{Outline of the paper.} The proof is structured as follows. In \textbf{Section 2}, we conformally change $(u_\alpha, W_\alpha)$ on $(M,g)$ to $(v_\alpha, Z_\alpha)$ defined in a Euclidean domain. In \textbf{Section 3}, we begin by obtaining pointwise estimates on both $v_\alpha$, $\nabla v_\alpha$, $\nabla^2 v_\alpha$ and $\mathcal{L}_g Z_\alpha$. \textbf{Section 4} begins with an immediate consequence of the aforementioned bounds: they yield a Harnack inequality on $v_\alpha$. Green's representation theory, applied to the elliptic operators of both the first and second equation, plays a central role in both obtaining and improving the aforementioned weak bounds on $\mathcal{L}_\xi Z_\alpha$.  The next step consists of using the techniques of asymptotic analysis to describe potential blow-up behaviour, and their interactions. All leads to a contradiction. 
\bigskip
\par\textbf{Acknowledgements.} I would like to extend my heartfelt thanks to Olivier Druet for his insight and continued support throughout the writing of this paper.
\section{Conformal changes of coordinates.}
Since $\left(M,g\right)$ is assumed to be locally conformally flat, for any sequence $x_\alpha\in M$ with $x_\alpha\to x$ as $\alpha\to +\infty$ and for any $\delta>0$ small enough, there exist smooth diffeomorphisms
\begin{equation}
\Phi_\alpha : U_\alpha\subset M \mapsto B_0\left(\delta\right)\subset {\mathbb R}^n
\end{equation}
and $\varphi_\alpha\in {\mathcal C}^\infty\left(B_0\left(\delta\right)\right)$ where $U_\alpha$ is some neighbourhood of $x_\alpha$ in $M$ such that 
\begin{equation}
\left(\Phi_\alpha^{-1}\right)^\star g = \varphi_\alpha(x)^{q-2}\xi
\end{equation}
where $\xi$ is the Euclidean metric. Moreover we can choose the diffeomorphisms $\Phi_\alpha$ and the functions $\varphi_\alpha$ to be uniformly bounded in any ${\mathcal C}^k$ for $k\le m$, $m$ fixed as we want. Note that we can also choose $\varphi_\alpha(0)=1$ and $\nabla\varphi_\alpha(0)=0$. 
For $x\in B_0(\delta)$, consider the change of functions
\begin{equation}
v_\alpha(x)=\varphi_\alpha(x)u_\alpha\circ\Phi_\alpha^{-1}(x)\quad\text{ and }\quad Z_\alpha(x)=\varphi_\alpha(x)^{2-q}\left(\Phi_\alpha\right)_*W_\alpha(x).
\end{equation}
This change of functions will be used repeatedly in the sequel.
First of all, note that, by \eqref{varepsilon2}, there exists $\varepsilon'>0$ such that
\begin{equation}\label{v varepsilon}
v_\alpha\geq \varepsilon'.
\end{equation}
Then it's convenient to recall the following formulas. Given that $\left(\Phi_\alpha^{-1}\right)^\star g_{ij}=\varphi_\alpha^{q-2}\xi_{ij}$, we see that the Laplace-Beltrami operator becomes
\begin{equation}
\begin{array}{r l}
\Delta_\xi v_\alpha & =\Delta_\xi\left(\varphi_\alpha u_\alpha\circ \Phi_\alpha^{-1}\right) \\
 &\quad\quad =\varphi_\alpha^{q-1}(x)\left(\Delta_g u_\alpha+\frac{n-2}{4(n-1)}R(g)u_\alpha\right)\left(\Phi_\alpha^{-1}\right)
\end{array}
\end{equation}
and that 
\begin{equation}
\varphi_\alpha^{q-2}\mathcal{L}_\xi Z_\alpha=\left(\Phi_\alpha\right)_* \left(\mathcal{L}_g W_\alpha\right).
\end{equation}
At last, the Lam\'e type operator transforms as
\begin{equation}
\overrightarrow{\Delta}_\xi\left(\varphi_\alpha^{2-q}\left(\Phi_\alpha\right)_*W_\alpha\right)_i-q\xi^{kl}\partial_k(\ln \varphi_\alpha)\mathcal{L}_\xi\left(\varphi_\alpha^{2-q}(\Phi_\alpha)_* W_\alpha\right)_{il}=\left(\Phi_\alpha\right)_*\left(\overrightarrow{\Delta}_g W_\alpha\right)_i,
\end{equation}
so
\begin{equation}
\left(\overrightarrow{\Delta}_\xi Z_\alpha \right)_i-q\left\langle\nabla\ln\varphi_\alpha,\mathcal{L}_\xi Z_\alpha \right\rangle_i=(\Phi_\alpha)_\star\left(\overrightarrow{\Delta}_g W_\alpha\right)_i.
\end{equation}
Simple but tedious computations lead then to the transformation of the system \eqref{perturbed system2} into 
\begin{equation}\label{conformal system}
\begin{array}{r l}
\Delta_\xi v_\alpha(x)+\tilde{h}_\alpha(x) v_\alpha(x)=& 
\tilde{f}_\alpha(x) v_\alpha^{q-1}(x)+\frac{\tilde{a}_\alpha(x)}{ v_\alpha^{q+1}(x)}- \frac{\tilde{b}_\alpha(x)}{v_\alpha(x)} \\
& -\langle\nabla v_\alpha(x),\tilde{Y}_\alpha(x)\rangle \left(\frac{\tilde{c}_\alpha(x)}{v_\alpha^2(x)}+\frac{\tilde{d}_\alpha(x)}{v_\alpha^{q+2}(x)}\right)\\
& -\frac{\langle\nabla v_\alpha(x),\tilde{Y}_\alpha(x)\rangle^2}{ v^{q+3}_\alpha(x)}\\
\left(\overrightarrow{\Delta}_\xi Z_\alpha\right)_i=&q\langle\nabla\ln\varphi_\alpha,\mathcal{L}_\xi Z_\alpha \rangle_i+\tilde{\mathcal{R}}_\alpha(v_\alpha,\nabla v_\alpha,\nabla^2 v_\alpha,\mathcal{L}_\xi Z_\alpha)_i
\end{array}
\end{equation}
where 
\begin{equation}\label{conformal system parameters}
\begin{array}{c}
\tilde{Y}_\alpha=\varphi_\alpha^2(\Phi_\alpha)_*Y_\alpha,\quad \tilde{f}_\alpha=f_\alpha\circ\Phi_\alpha^{-1},\\ \tilde{h}_\alpha=\varphi_\alpha^{q-2}\left(h_\alpha-\frac{n-2}{4(n-1)}R(g)\right)\circ \Phi_\alpha^{-1},\\ 
\tilde{b}_\alpha=\varphi_\alpha^{q}b_\alpha\circ\Phi_\alpha^{-1}-\varphi_\alpha\langle\nabla \varphi_\alpha,(\Phi_\alpha)_\star Y_\alpha\rangle c_\alpha\circ\Phi_\alpha^{-1},\\
\tilde{c}_\alpha=c_\alpha\circ\Phi_\alpha^{-1},\quad \tilde{d_\alpha}=2\varphi_\alpha\langle\nabla\varphi_\alpha,(\Phi_\alpha)_\star Y_\alpha\rangle+\varphi_\alpha^q d_\alpha\circ \Phi_\alpha^{-1},\\
\tilde{a}_\alpha=\tilde{\rho}_{1,\alpha}+\left|\tilde{\Psi}_\alpha+\tilde{\rho}_{2,\alpha}\mathcal{L}_\xi Z_\alpha\right|^2_\xi\\
\tilde{\rho}_{1,\alpha}=\varphi_\alpha^{2q}\rho_{1,\alpha}\circ \Phi_\alpha^{-1}+\varphi_\alpha^{q+1}\langle\nabla\varphi_\alpha,(\Phi_\alpha)_\star Y_\alpha \rangle d_\alpha\circ \Phi_\alpha^{-1}-\varphi_\alpha^2\langle\nabla \varphi_\alpha, (\Phi_\alpha)_\star, Y_\alpha \rangle^2,\\ \tilde{\rho}_{2,\alpha}=\varphi_\alpha^q\rho_{2,\alpha}\circ\Phi_\alpha^{-1},\quad
\tilde{\Psi}_\alpha=\varphi_\alpha^2\left(\Phi_\alpha\right)_*\Psi_\alpha
\end{array}
\end{equation}
and 
\begin{equation}\label{conformal system parameters-bis}
\begin{array}{rcl}
\widetilde{\mathcal{R}}_\alpha(v_\alpha,\nabla v_\alpha,\nabla^2 v_\alpha,\mathcal{L}_\xi Z_\alpha) &=& \left(\Phi_\alpha\right)_\star \mathcal{R}_\alpha(u_\alpha,\nabla u_\alpha,\nabla^2 u_\alpha,\mathcal{L}_g W_\alpha)\\ \\
&=&\langle \left(\Phi_\alpha\right)_\star \nabla \ln\tilde{N},{\mathcal L}_\xi Z_\alpha\rangle \\
&&+2\frac{n-1}{n-2}\varphi_\alpha^{2-q}\Big(\frac{3n-2}{n-2}\frac{\langle \nabla v_\alpha, \left(\Phi_\alpha\right)_\star \tilde{V}_\alpha\rangle \nabla v_\alpha}{v_\alpha^2} \\
&&\quad\quad\quad\quad\quad\quad\quad\quad- \frac{\langle \nabla^2 v_\alpha, \left(\Phi_\alpha\right)_\star \tilde{V}_\alpha\rangle}{v_\alpha}\Big)\\
&&+ \widetilde{\mathcal{T}}_\alpha(v_\alpha,\nabla v_\alpha, \nabla^2 v_\alpha).
\end{array}
\end{equation}
Here, $\widetilde{\mathcal{T}}_\alpha$ denotes the lower order terms of the second equation. It is clear that we have
\begin{equation}
\vert \widetilde{\mathcal{T}}_\alpha(v_\alpha,\nabla v_\alpha,\nabla^2 v_\alpha)\vert\leq C\left(1+\vert\vert(\Phi_\alpha)_\star\tilde{V}_\alpha\vert\vert_{\mathcal{C}^1} \left\vert\frac{\nabla v_\alpha}{v_\alpha}\right\vert \right).
\end{equation}
\section{Weak pointwise estimates}
The following result describes a pointwise estimate that holds everywhere on $M$. It provides a way to identify a set of points $\mathcal{S}_\alpha$ where $u_\alpha$ or $\mathcal{L}_g W_\alpha$ can potentially explode.
\begin{lemma}\label{lem: first estimate}
Let $(u_\alpha, W_\alpha)$ be a sequence of solutions of the perturbed system \eqref{perturbed system2}, verifying the non-compactness hypothesis \eqref{hyp to contradict2}. There exists an integer $N_\alpha\in\N^*$ and a set of critical points  $\mathcal{S}_\alpha=(x_{1,\alpha},\dots, x_{N_\alpha,\alpha})$ of $u_\alpha$ such that
\begin{equation}
d_g(x_{i,\alpha},x_{j,\alpha})^{n} u_\alpha(x_{i,\alpha})^q\geq 1,
\end{equation}
for all $1\leq i,j\leq N_\alpha$, $i\not=j$, and
\begin{equation}
\Big(\min_{1\leq i\leq N_\alpha}d_g(x_{i,\alpha},x)\Big)^n u_\alpha(x)\leq 1
\end{equation}
for any $x$ critical point of $u_\alpha$ in $M$, and
\begin{equation}
\begin{array}{r l}
 \Big(\min_{1\leq i\leq N_\alpha}&d_g(x_{i,\alpha},x)\Big)^n  \\
&\times \left(u_\alpha^q(x)+\left\vert\frac{\nabla u_\alpha(x)}{u_\alpha(x)}\right\vert^n+\left\vert\frac{\nabla^2u_\alpha(x)}{u_\alpha(x)}\right\vert^\frac{n}{2}+\vert\mathcal{L}_g W_\alpha\vert_g(x)\right)\leq C.
\end{array}
\end{equation}
\end{lemma}
\begin{proof}
\par\textbf{Step $1$: Setting up the proof by contradiction.} For every $\alpha\in\N^*$, we may define the integer $N_\alpha\in\N^*$ and the set of critical points  $$\mathcal{S}_\alpha=(x_{1,\alpha},\dots, x_{N_\alpha,\alpha})$$ of $u_\alpha$ by the following lemma, which holds very generally for any sufficiently regular function.
\begin{lemma}\label{lem: former lemma 12}
Let $u$ be a positive real-valued $\mathcal{C}^2$ function defined in a compact manifold $M$. Then there exists $N\in \N^*$ and $\left(x_1,\, x_2,\dots x_N\right)$ a set of critical points of $u$ such that
\begin{equation}
d_g(x_i,x_j)^{\frac{n-2}{2}}u(x_i)\geq 1
\end{equation}
for all $i,\,j\in\{1, \dots, N\}$, $i\not= j$, and 
\begin{equation}\label{the critical points hypothesis}
\left(\min_{i=1,\dots, N}d_g(x_i,x)\right)^{\frac{n-2}{2}}u(x)\leq 1
\end{equation}
for all critical points $x$ of $u$.
\end{lemma}
The lemma and its proof may be found in Druet and Hebey's paper \cite{DruHeb09}. Let
\begin{equation}\label{scaling invariant quantity}
\begin{array}{r l}
\Psi_\alpha(x)=& \Big(\min_{1\leq i\leq N_\alpha}d_g(x_{i,\alpha},x)\Big)^n  \\
&\times \left(u^q_\alpha(x)+\left\vert\frac{\nabla u_\alpha(x)}{u_\alpha(x)}\right\vert^n+\left\vert\frac{\nabla^2u_\alpha(x)}{u_\alpha(x)}\right\vert^\frac{n}{2}+\vert\mathcal{L}_g W_\alpha\vert_g(x)\right)
\end{array}
\end{equation}
for $x\in M$. Let $(x_\alpha)_\alpha\in M$ be such that
\begin{equation}\label{definition of x alpha}
\Psi_\alpha(x_\alpha)=\sup_M\Psi_\alpha\to\infty
\end{equation}
as $\alpha\to\infty$. 
\par\textbf{Step $2$: Rescaling.} We denote the injectivity radius of $M$ by $i_g(M)$. Let \begin{equation}
0<\delta_\alpha<\frac{1}{2}i_g(M)
\end{equation}
be radii around $x_\alpha$. Since $(M,g)$ is conformally flat, let $\varphi_\alpha$ and $\Phi_\alpha$ be as in previous section
so that $\left(\Phi_\alpha^{-1}\right)^\star g_{ij}=\varphi_\alpha^{q-2}\xi_{ij}$, $\varphi_\alpha(0)=1$ and $\nabla\varphi_\alpha(0)=0$. In fact, these conformal factors can be chosen to be uniformly bounded up in $\mathcal{C}^k$, up to an arbitrary $k>0$. Consider the following rescalings of the conformal factors:
\begin{equation}\label{rescallings2}
\begin{array}{r l} \hat{v}_\alpha(x)=&\mu_\alpha^{\frac{n-2}{2}}\varphi_\alpha(\mu_\alpha x)u_\alpha\circ\Phi_\alpha^{-1}(\mu_\alpha x)\\ 
 \hat{Z}_\alpha(x)=&\mu_\alpha^{n-1}\varphi_\alpha(\mu_\alpha x)^{2-q}\left(\Phi_\alpha\right)_* W_\alpha(\mu_\alpha x),
\end{array}
\end{equation}
where $x\in\Omega_\alpha$, with $\Omega_\alpha:=B_{x_\alpha}\left(\frac{\delta}{\mu_\alpha}\right)$ and
\begin{equation}\label{mu2}
\mu_\alpha^{-n}:=u_\alpha(x_\alpha)^q+\left\vert\frac{\nabla u_\alpha(x_\alpha)}{u_\alpha(x_\alpha)}\right\vert^n+\left\vert\frac{\nabla^2u_\alpha(x_\alpha)}{u_\alpha(x_\alpha)}\right\vert^\frac{n}{2}+\vert\mathcal{L}_g W_\alpha\vert_g(x_\alpha).
\end{equation}
Moreover, because $M$ is compact, and by \eqref{scaling invariant quantity} and \eqref{definition of x alpha},
\begin{equation}\label{est sizes}
\frac{d_g(x_\alpha,\mathcal{S}_\alpha)}{\mu_\alpha}\to\infty\quad\text{and}\quad\mu_\alpha\to 0.
\end{equation}
We consider the rescaled perturbed system corresponding to \eqref{rescallings2},
\begin{equation}\label{E alpha}
\begin{array}{r l}
 \Delta_\xi\hat{v}_\alpha=
& -\mu_\alpha^2\hat{h}_\alpha\hat{v}_\alpha+\hat{f}_\alpha\hat{v}_\alpha^{q-1}+\frac{\hat{a}_\alpha}{\hat{v}_\alpha^{q+1}} -\mu_\alpha^{\frac{n+2}{2}}\frac{\hat{b}_\alpha}{v_\alpha(\mu_\alpha\cdot)} \\ 
& -\mu_\alpha^{\frac{n}{2}}\frac{\langle\nabla\hat{v}_\alpha,\hat{Y}_\alpha\rangle}{\hat{v}_\alpha}\left(\frac{\hat{d}_\alpha}{v_\alpha(\mu_\alpha\cdot)}+\frac{\hat{c}_\alpha}{v_\alpha^{q+1}(\mu_\alpha\cdot)}\right)\\
& -\mu_\alpha^{\frac{n-2}{2}}\frac{\langle\nabla \hat{v}_\alpha,\hat{Y}_\alpha\rangle^2}{\hat{v}^2_\alpha}\frac{1}{v_\alpha^{q+1}(\mu_\alpha\cdot)} \\
 \overrightarrow{\Delta}_\xi \hat{Z}_{\alpha }=& q\mu_\alpha\xi^{kl}\partial_k(\ln \varphi_\alpha)(\mu_\alpha\cdot)\mathcal{L}_\xi \hat{Z}_{\alpha l}
 \\&\quad\quad+\hat{\mathcal{R}}_\alpha(\hat{v}_\alpha,\nabla\hat{v}_\alpha,\nabla^2\hat{v}_\alpha, \mathcal{L}_\xi \hat{Z}_\alpha),
\end{array}
\end{equation}
where
\begin{equation}\label{rescaled parameters}
\begin{array}{c}
\hat{h}_\alpha(x)=\tilde{h}_\alpha(\mu_\alpha x),\quad
\hat{f}_\alpha(x)=\tilde{f}_\alpha(\mu_\alpha x),\\
\hat{\rho}_{1,\alpha}(x)=\tilde{\rho}_{1,\alpha}(\mu_\alpha x),\quad
\hat{\rho}_{2,\alpha}(x)=\tilde{\rho}_{2,\alpha}(\mu_\alpha x),\\
\hat{\Psi}_\alpha(x)=\tilde{\Psi}_\alpha(\mu_\alpha x), \quad
\hat{Y}_\alpha(x)=\tilde{Y}_\alpha(\mu_\alpha x),\\
\hat{a}_\alpha(x)=\tilde{a}_\alpha(\mu_\alpha x),\\
\hat{c}_\alpha(x)=\tilde c_\alpha(\mu_\alpha x),\quad
\hat{d}_\alpha(x)=\tilde{d}_\alpha(\mu_\alpha x).
\end{array}
\end{equation}
and 
\begin{equation}
\begin{array}{r l}
 \hat{\mathcal{R}}(\hat{v}_\alpha,\nabla\hat{v}_\alpha,\nabla^2\hat{v}_\alpha,\mathcal{L}_\xi Z_\alpha)\leq & C_\mathcal{R}'\Big(\mu_\alpha^{n+1}+\mu_\alpha^n\left\vert\frac{\nabla \hat{v}_\alpha}{\hat{v}_\alpha}\right\vert+\mu_\alpha^{n-1} \left\vert\frac{\nabla \hat{v}_\alpha}{\hat{v}_\alpha}\right\vert^2\\
&\quad\quad +\mu_\alpha^{n-1}\left\vert\frac{\nabla^2\hat{v}_\alpha}{\hat{v}_\alpha}\right\vert+\mu_\alpha\vert\mathcal{L}_\xi \hat{Z}_\alpha \vert\Big).
\end{array}
\end{equation}
where $\mathcal{C}'_\mathcal{R}$ is a constant.

\medskip
By the definition \eqref{mu2},
\begin{equation}\label{rescaled mu}
\hat{v}_\alpha(0)^q+\left\vert\frac{\nabla \hat{v}_\alpha(0)}{\hat{v}_\alpha(0)}\right\vert^n+\left\vert\frac{\nabla^2\hat{v}_\alpha(0)}{\hat{v}_\alpha(0)}\right\vert^\frac{n}{2}+\vert\mathcal{L}_\xi\hat{Z}_\alpha\vert_\xi(0)=1
\end{equation}
and for any $R>0$, 
\begin{equation}\label{rescaled est}
\sup_{x\in B_0(R)}\left(\hat{v}_\alpha^q(x)+\left\vert\frac{\nabla \hat{v}_\alpha(x)}{\hat{v}_\alpha(x)}\right\vert^n+\left\vert\frac{\nabla^2\hat{v}_\alpha(x)}{\hat{u}_\alpha(x)}\right\vert^\frac{n}{2}+\left\vert\mathcal{L}_\xi\hat{Z}_\alpha(x)\right\vert_\xi\right)\leq 1+o(1)
\end{equation}
and thereby
\begin{equation}
\sup_{B_0(R)}\vert\nabla \ln\hat{v}_\alpha(x)\vert\leq 1 + o(1)
\end{equation}
for all $R>0$. As a consequence,
\begin{equation}
\hat{v}_\alpha(0)e^{-2|x|}\leq \hat{v}_\alpha(x)\leq \hat{v}_\alpha(0)e^{2|x|}.
\end{equation}
\par\textbf{Step $3$: $\vert \mathcal{L}_\xi \hat{Z}_\alpha\vert$ converges to zero.} By Green's representation formula applied to the first equation of \eqref{E alpha} on $B_x(3R)$, we get
\begin{equation}\label{green formula for u hat}
\begin{array}{r l}
 \hat{v}_\alpha(x)& \geq \int_{B_x(3R)}\mathcal{G}_{3R}(x,y)\Big[ \frac{\hat{a}_\alpha(y)}{\hat{v}_\alpha^{q+1}(y)} -\mu_\alpha^2\hat{h}_\alpha(y)\hat{v}_\alpha(y)-\mu_\alpha^{\frac{n+2}{2}}\frac{\hat{b}_\alpha(y)}{v_\alpha(\mu_\alpha y)}  \\
& \quad\quad\quad\quad\quad -\mu_\alpha^{\frac{n}{2}} \frac{\langle\nabla\hat{v}_\alpha(y),\hat{Y}_\alpha(y)\rangle}{\hat{v}_\alpha(y)}\left(\frac{\hat{d}_\alpha(y)}{v_\alpha(\mu_\alpha y)}+\frac{\hat{c}_\alpha(y)}{v_\alpha^{q+1}(\mu_\alpha y)}\right)\\
& \quad\quad\quad\quad\quad -\mu_\alpha^{\frac{n-2}{2}} \frac{\langle\nabla \hat{v}_\alpha(y),\hat{Y}_\alpha(y)\rangle^2}{\hat{v}_\alpha(y)}\frac{1}{v_\alpha^{q+1}(\mu_\alpha y)}\Big] \,dy.\\
\end{array}
\end{equation}
Here, $\mathcal{G}_{3R}(x,y):=\frac{1}{(n-2)\omega_{n-1}}\left(\vert x-y\vert^{2-n}-(3R)^{2-n}\right)$. By taking $\alpha$ large, we get the bulk integral estimate
\begin{equation}
\int_{B_x(2R)}|x-y|^{2-n}|\mathcal{L}_\xi\hat{Z}_\alpha|^2_\xi(y)\,dy\leq C,
\end{equation}
where $C$ a positive constant independent of $R$ or $\alpha$. Therefore, we may find $s_\alpha\in(\frac{3}{2}R,2R)$ such that the boundary estimate 
\begin{equation}
\int_{\partial B_0(s_\alpha)}\vert\mathcal{L}_\xi \hat{Z}_\alpha\vert^2_\xi(y)\,d\sigma(y)\leq CR^{n-3}
\end{equation}
holds. Moreover,
\begin{equation}
\left\vert q\mu_\alpha\xi^{kl}\partial_k\left(\ln \varphi_\alpha\right)(\mu_\alpha\cdot)\left(\mathcal{L}_\xi\hat{Z}_\alpha\right)_{li}\right\vert\leq C\mu_\alpha^2\vert y\vert \left\vert\mathcal{L}_\xi\hat{Z}_\alpha\right\vert_\xi.
\end{equation}
Turning to the second equation of \eqref{E alpha}, we use the Green representation formula for the Lam\'e type operator $\overrightarrow{\Delta}_\xi$ in $B_0(2R)$. This yields
\begin{equation}\label{the green function approximation}
\begin{array}{r l}
 \left\vert\mathcal{L}_\xi\hat{Z}_\alpha\right\vert_\xi(x)
 \leq &  C\int_{B_0(s_\alpha)}\vert x-y\vert^{1-n}\left\vert \overrightarrow{\Delta}_\xi \hat{Z}_\alpha\right\vert\,dy \\ 
& + C\int_{\partial B_0(s_\alpha)}\vert x-y\vert^{1-n}\left\vert\mathcal{L}_\xi\hat{Z}_\alpha\right\vert_\xi(y)\,d\sigma (y)\\ 
\leq&  C'R\mu_\alpha+\frac{C'}{R}.
\end{array}
\end{equation}
for positive constants $C$ and $C'$. We therefore get an improvement on the pointwise estimate of the rescaled $\hat{W}_\alpha$ from \eqref{rescaled est}:
\begin{equation}\label{saving estimate}
\left\vert\mathcal{L}_\xi \hat{Z}_\alpha\right\vert_\xi\to 0
\end{equation}
and 
\begin{equation}\label{a est}
\hat{a}_\alpha\to 0
\end{equation}
in $\mathcal{C}^0_{loc}(\R^n)$ as $\alpha\to\infty$.
\par\textbf{Step $4$: The study of potential blow-up profiles.} We turn to the study of the remaining terms of \eqref{mu2}. From (\ref{rescaled mu}) and (\ref{saving estimate}), we deduce that
\begin{equation}\label{bounds on u hat}
\lim_{\alpha\to\infty} \left( \hat{v}^q_\alpha(0)+\left\vert\frac{\nabla \hat{v}_\alpha (0)}{\hat{v}_\alpha(0)}\right\vert^n+\left\vert\frac{\nabla^2\hat{v}_\alpha(0)}{\hat{v}_\alpha(0)}\right\vert^\frac{n}{2}\right)= 1.
\end{equation}  
Let us set
\begin{equation}\label{def w alpha2}
w_\alpha(x):=\frac{\hat{v}_\alpha(x)}{\hat{v}_\alpha(0)}=\frac{u_\alpha(\exp_{x_\alpha}(\mu_\alpha x))}{u_\alpha(x_\alpha)}.
\end{equation}
It follows that
\begin{equation}\label{w cst}w_\alpha(0)=1\quad\text{and }\quad
\lim_{\alpha\to\infty}\left( \left\vert\frac{\nabla w_\alpha(0)}{w_\alpha(0)}\right\vert^n+\left\vert\frac{\nabla^2 w_\alpha(0)}{w_\alpha(0)} \right\vert^{\frac{n}{2}}\right) \leq 1
\end{equation}
and therefore
\begin{equation}
e^{-2\vert x\vert}\leq w_\alpha(x)\leq e^{2\vert x\vert}
\end{equation}
We divide the first equation of system (\ref{E alpha}) by $\hat{u}_\alpha(0)$ and obtain
\begin{equation*}\label{w eq}
    \begin{array}{r l}
     \Delta_\xi w_\alpha=     & -\mu_\alpha^2\hat{h}_\alpha w_\alpha + \hat{f}_\alpha w_\alpha^{q-1}\hat{v}_\alpha^{q-2}(0)   \\
     &+\frac{\mu_\alpha^{\frac{n+2}{2}}\hat{\rho}_{1,\alpha}(x)}{\hat{v}^{q+1}_\alpha(x)v_\alpha(x_\alpha)}+\frac{\vert\mu_\alpha^{\frac{n+2}{4}}\hat{\Psi}_{\alpha}(x)+\hat{\rho}_{2,\alpha}(x)\mathcal{L}_\xi \hat{Z}_{\alpha}(x)\vert^2_\xi}{\hat{v}^{q+1}_\alpha(x)v_\alpha(x_\alpha)}
     \\
         & -\mu_\alpha^2\frac{\hat{b}_\alpha}{v_\alpha\left(\exp_{x_\alpha}(\mu_\alpha\cdot)\right)v_\alpha(x_\alpha)}
          -\frac{\langle\nabla w_\alpha,\hat{Y}_\alpha\rangle^2}{w_\alpha^{2}}\frac{1}{v_\alpha(x_\alpha)v_\alpha^{q+1}(\mu_\alpha\cdot)}\\
         & -\mu_\alpha \langle\nabla w_\alpha,\hat{Y}_\alpha\rangle\left(\frac{\hat{d}_\alpha }{v_\alpha\left(\exp_{x_\alpha}(\mu_\alpha\cdot)\right)v_\alpha(x_\alpha)} 
          +\frac{\hat{c}_\alpha}{v_\alpha^{q+1}\left(\exp_{x_\alpha}(\mu_\alpha\cdot)\right)v_\alpha(x_\alpha)}\right).
    \end{array}
\end{equation*}
Up to a subsequence, we denote
\begin{equation}
\begin{array}{r l}
 \hat{l}_\alpha=\hat{v}_\alpha(0), &\quad\text{with}\quad \lim_{\alpha\to\infty}\hat{l}_\alpha=:\hat{l}\in [0,1],\\
 l_\alpha=v_\alpha^{-1}(x_\alpha),&\quad\text{with}\quad \lim_{\alpha\to\infty}l_\alpha=:l\in [0,\varepsilon^{-1}]
\end{array}
\end{equation}
which follows from \eqref{rescallings2} and \eqref{bounds on u hat} in the case of the first limit, and from \eqref{varepsilon2} for the second. Furthermore, \eqref{mu2} implies that
\begin{equation}\label{diff cases2}
l\hat{l}=\lim_{\alpha\to\infty}\mu_\alpha^{\frac{n-2}{2}}=0.
\end{equation}
\begin{rem}
It is here that we use the hypothesis $V_\alpha(x_\alpha)\to 0$.
\par We denote
\begin{equation}
\frac{\mathcal{L}_\xi\hat{Z}_\alpha(x)}{\hat{l}^{\frac{q+2}{2}}_\alpha}\to \mathcal{L}_\xi Z.
\end{equation}
\par By standard elliptic theory, we find that there exists $w:=\lim_{\alpha\to\infty}w_\alpha$ in $\mathcal{C}^{1,\eta}_{\text{loc}}(\R^n)$, and by dividing the first equation by $\hat{l}$, we obtain
\begin{equation}
\begin{array}{r l}
\Delta w=& \frac{n-2}{4(n-1)}\left(2V(\psi(x_0))-\frac{n-1}{n}\tau^{*2}\right)w^{q-1}\hat{l}^{q-2} +\frac{n-2}{16(n-1)}\frac{\tilde{N}^2(x_0)\vert\mathcal{L}_\xi Z\vert^2_\xi}{w^{q+1}}\\
 &-\frac{n}{n-2}\frac{\langle\nabla w, \tilde{N}(x_0)\tilde{V}(x_0)\rangle^2}{w^{q+3}}l^{\,q+2}\\
 \overrightarrow{\Delta}_\xi Z=&-2\frac{n-1}{n+1}\frac{\langle\tilde{V}(x_0),\nabla w\rangle\nabla w}{w^2}l^{\frac{q+2}{2}}-\frac{n-1}{n}\langle\tilde{V}(x_0),\frac{\nabla^2 w}{w}\rangle l^{\frac{q+2}{2}}.
\end{array}
\end{equation}
Since $\tilde{V}(x_0)=0$, we obtain $\mathcal{L}_\xi Z=0$. Had we not imposed this hypothesis, the next step would have been to classify the solutions of the second equation. To our knowledge, this is an open problem.
\end{rem}
Therefore, the limit equation becomes:
$$\Delta w= f(x_0)w^{q-1}\hat{l}^{q-2}.$$
In fact, we can easily tackle the slightly more general equation 
\begin{equation}\label{limiting equation2}
\Delta w=f(x_0)w^{q-1}\hat{l}^{q-2}-\frac{\langle\nabla w, Y(x_0)\rangle^2}{w^{q+3}}l^{\,q+2},\quad x\in\R^n,
\end{equation}
even if $Y(x_0)\not=0$. Based on the observation \eqref{diff cases2}, we consider three separate cases.
\par\textbf{First case.} Let 
\begin{equation}\label{l zero l hat non zero2}
l=0\quad\text{ and }\hat{l}\not=0.
\end{equation}
Then by passing to the limit in the first equation of (\ref{E alpha}), we obtain
\begin{equation}\label{first limit eq2}
\Delta U=f(x_0)U^{q-1}
\end{equation}
in $\R^n$. The exact form of the solutions of this equation is known, thanks to the work of Caffarelli, Gidas and Spruck \cite{CafGidSpr89}:
\begin{equation}
U(x)=\left(1+\frac{f(x_0)|x-y_0|^2}{n(n-2)}\right)^{1-\frac{n}{2}}\quad\text{ or }\quad U\equiv 0,
\end{equation}
\par If $U$ is non-trivial, with $y_0\in \R^n$ the unique maximum point, there exist $(y_\alpha)_\alpha$ local maxima of $(v_\alpha)_\alpha$ approaching $y_0$ such that 
\begin{equation}\label{est: blow up analysis, distances2}
    d_g(x_\alpha,y_\alpha)=O(\mu_\alpha)
\end{equation}
and
\begin{equation}\label{est lim 12}
    \mu_\alpha^{\frac{n-2}{2}}v_\alpha(y_\alpha)\to 1\quad \text{ as }\quad \alpha\to\infty.
\end{equation}
Since $(y_\alpha)_\alpha$ are critical points, the hypothesis \eqref{the critical points hypothesis} implies that
$$d_g(\mathcal{S}_\alpha, y_\alpha)^{\frac{n-2}{2}}v_\alpha(y_\alpha)\leq 1$$
for all $\alpha\in \N$, so by (\ref{est lim 12}), $\displaystyle d_g(\mathcal{S}_\alpha,y_\alpha)=O(\mu_\alpha)$; together with (\ref{est: blow up analysis, distances2}), the triangle inequality implies $\displaystyle d_g(\mathcal{S}_\alpha,x_\alpha)=O(\mu_\alpha)$, which contradicts (\ref{est sizes}).
\par If $U\equiv 0$, then 
\begin{equation}
\lim_{\alpha\to\infty}\hat{v}_\alpha(0)=0,
\end{equation} 
which contradicts \eqref{l zero l hat non zero2}.
\par\textbf{Second case.} Let 
\begin{equation}\label{l non zero l hat zero2}
l\not=0\quad\text{ and }\quad\hat{l}=0.
\end{equation}
Since $l\not=0$, thanks to \eqref{varepsilon2} and \eqref{def w alpha2}, $w$ is bounded from below by a constant,
\begin{equation}
w\geq \varepsilon l.
\end{equation}
Note also that \eqref{limiting equation2} implies that $w$ is subharmonic and that
\begin{equation}
\Delta w^{-\alpha}\leq \alpha\frac{\vert\nabla w\vert^2}{w^{\alpha+2}}\left[\frac{\vert Y(x_0)\vert^2}{\varepsilon^{q+2}}-(\alpha+1)\right],
\end{equation}
so $w^{-\alpha}$ is subharmonic for $\alpha$ large. By applying Lemma \ref{lem: A12} (see the Annex), we deduce that $w$ is constant, in contradiction with \eqref{w cst}.
\par\textbf{Third case.} Let 
\begin{equation}\label{l zero l hat zero2}
l=0\quad\text{ and }\quad\hat{l}=0,
\end{equation}
then $w$ is a non-negative harmonic function on $\R^n$. Thus, $w=cst$ and furthermore, by \eqref{w cst}, $w\equiv 1$. But $\hat{l}=0$ implies that 
\begin{equation}
\left\vert\frac{\nabla w(0)}{w(0)}\right\vert^n+\left\vert\frac{\nabla^2 w(0)}{w(0)}\right\vert^\frac{n}{2}=1.
\end{equation}
which leads to a contradiction.
\end{proof}

\section{Asymptotic analysis}
In this section, we assume that $(u_\alpha,W_\alpha)_{\alpha\in\N}$ is an $L^\infty$ blow-up sequence, \textit{i.e.} we ask that there exist a sequence $(x_\alpha)_{\alpha\in\N}$ of critical points of $(u_\alpha)_{\alpha\in\N}$ and a series of positive real numbers $(\rho_\alpha)_{\alpha\in\N}$, where
\begin{equation}
0<\rho_\alpha<\frac{1}{16}i_g(M),
\end{equation}
such that
\begin{equation}\label{est: blow up no 22}
\rho_\alpha^n\sup_{B_{x_\alpha}(8)}u_\alpha^q(\rho_\alpha x)\to \infty\quad\text{ as }\alpha\to\infty,
\end{equation}
and moreover we ask that 
\begin{equation}\label{est: blow up no 32}
\begin{array}{c}
d_g(x_\alpha,x)^n\left(u_\alpha^q(x)+\left|\frac{\nabla u_\alpha(x)}{u_\alpha(x)}\right|^n+\left\vert\frac{\nabla^2u_\alpha(x)}{u_\alpha(x)}\right\vert^\frac{n}{2}+|\mathcal{L}_g W_\alpha|_g(x)\right)\leq C, \\
\quad\quad\quad\quad\quad\quad\quad\quad\quad\quad\quad\quad\quad\quad\quad\quad x\in B_{x_\alpha}(8\rho_\alpha).
\end{array}
\end{equation} 
In the reminder of this section, we assume that $\left(u_\alpha, W_{\alpha}\right)_{\alpha\in\N}$  is a blow-up sequence, and we look at the kind of asymptotic profiles we can potentially obtain. At the very end, we rule all of them out, and thus obtain our compactness result. Note that, if we were to assume that \eqref{est: blow up no 22} holds for a sequence $(x_\alpha)_\alpha$ in $\mathcal{S}_\alpha$, with $\rho_\alpha$ smaller than the distance of $x_\alpha$ to any other point in $\mathcal{S}_\alpha$, then \eqref{est: blow up no 32} holds as well. 
\subsection{Harnack inequality}
The following is a Harnack-type inequality. It is a direct consequence of the weak estimate and it plays a key role in ruling out clusters of bubbles where some are much larger than others.
\begin{lemma}\label{harnack lemma}
Let $(u_\alpha, \rho_\alpha)_\alpha$ be a blow-up sequence such that \eqref{est: blow up no 22} and \eqref{est: blow up no 32} hold. Then there exists a constant $C_3>1$ such that for any sequence $0<s_\alpha\leq\rho_\alpha$, we get
\begin{equation}
s_\alpha^2\vert\vert\nabla^2 u_\alpha \vert\vert_{L^\infty(\Omega_\alpha)}+s_\alpha\vert\vert\nabla u_\alpha\vert\vert_{L^\infty(\Omega_\alpha)}\leq C_3\sup_{\Omega_\alpha}u_\alpha\leq C_3^2\inf_{\Omega_\alpha}u_\alpha,
\end{equation}
where $\Omega_\alpha=B_{x_\alpha}(6s_\alpha)\backslash B_{x_\alpha}(\frac{1}{6}s_\alpha).$
\end{lemma}
 
\begin{rem}
 When considering a rescaling of the type
\begin{equation}
\bar{u}(x)=s_\alpha^{\frac{n-2}{2}}u_\alpha(\exp_{x_\alpha}(s_\alpha x)),
\end{equation}
and $\bar{\Omega}_\alpha=B_0(6)\setminus B_0(\frac{1}{6})$, then the above lemma gives
\begin{equation}\label{harnack inequality}
\vert\vert\nabla^2\bar{u}_\alpha\vert\vert_{L^\infty(\bar{\Omega}_\alpha)}+\vert\vert \nabla \bar{u}_\alpha\vert\vert_{L^\infty(\bar{\Omega}_\alpha)}\leq C_3\sup_{\bar\Omega_\alpha}\bar{u}_\alpha\leq C_3^2\inf_{\bar\Omega_\alpha}\bar{u}_\alpha.
\end{equation}
\end{rem}
\begin{proof}[Proof of Lemma \eqref{harnack lemma}:]
Estimate (\ref{mu2}) implies that
\begin{equation}\label{e12}
\left|\frac{\nabla u_\alpha(x)}{u_\alpha(x)}\right|\leq C_2 d_g(x_\alpha,x)^{-1}\quad\text{in }\Omega_\alpha,
\end{equation}
and therefore
\begin{equation}\label{e22}
s_\alpha|\nabla\ln u_\alpha(x)|\leq 6C_2\quad\text{in }\Omega_\alpha.
\end{equation}
Similarly, it holds true that
\begin{equation}
s_\alpha^2|\nabla^2 \ln u_\alpha(x)|\leq 6C_2\quad\text{in }\Omega_\alpha.
\end{equation}
Taking $C_3\geq 6C_2$, we get the first inequality from \eqref{harnack inequality}.
Then, from (\ref{e22}) and from the fact that the domain is an annulus  $\Omega_\alpha=B_{x_\alpha}(6s_\alpha)\backslash B_{x_\alpha}(\frac{1}{6}s_\alpha)$, we estimate that
$$\sup_{\Omega_\alpha}\ln u_\alpha- \inf_{\Omega_\alpha}\ln u_\alpha\leq l_\alpha(\Omega_\alpha)||\nabla \ln u_\alpha||_{L^\infty(\Omega_\alpha)}\leq 42C_2,$$
where $l_\alpha(\Omega_\alpha)$ is the infimum of the length of a curve in $\Omega_\alpha$ drawn between a maximum and a minimum of $u_\alpha$.  Equivalently
$$\sup_{\Omega_\alpha} u_\alpha\leq e^{42C_2}\inf_{\Omega_\alpha} u_\alpha,$$
so it suffices to take $C_3=e^{42C_2}$.
\end{proof}
Let $(B_{x_\alpha}(16),\Phi_\alpha)$ be a conformal chart around $x_\alpha$. We study the blow-up sequence in a Euclidean framework through these charts. By the properties we've imposed on $\varphi_\alpha$, 
\begin{equation}
\left|q\xi^{kl}\partial_k\left(\ln \varphi_\alpha\right)\left(\mathcal{L}_\xi Z_\alpha\right)_{li}\right|\leq C|y||\mathcal{L}_\xi Z_\alpha|_\xi,\quad\text{on }B_0(8\rho_\alpha).
\end{equation}
By the definition of a blow-up sequence, we also get that
\begin{equation}\label{weak}
|x|^n\left( v_\alpha^q(x) +\left\vert\frac{\nabla v_\alpha(x)}{v_\alpha(x)}\right\vert^n+\left\vert\frac{\nabla^2v_\alpha(x)}{v_\alpha(x)}\right\vert^{\frac{n}{2}}+\vert\mathcal{L}_\xi Z_\alpha(x)\vert\right)\leq C.
\end{equation}

\subsection{Strong estimate on $v_\alpha$ in $B_0(\mu_\alpha)$}
The following result is a strong estimate on the size of a blow-up sequence in a very small ball $B_0(\mu_\alpha)$.
\begin{lemma}\label{first blow up}
Let $(u_\alpha, W_\alpha)_{\alpha\in\N}$ be a blow-up sequence. Let
\begin{equation}\label{mu v}
\mu_\alpha^{1-\frac{n}{2}}:=u_\alpha(x_\alpha)=v_\alpha(0).
\end{equation}
Up to a subsequence, we have 
\begin{equation}\label{mu and rho}
\mu_\alpha\to 0\quad\text{ and }\quad\frac{\rho_\alpha}{\mu_\alpha}\to\infty.
\end{equation}
Moreover, we see that
\begin{equation}\label{v mu}
\mu_\alpha^{\frac{n-2}{2}}v_\alpha(\mu_\alpha x)\to U(x)\quad\text{ in }\mathcal{C}^{2,\eta}_{loc}(\R^n)\quad\text{ as }\alpha\to\infty
\end{equation}
and
\begin{equation}\label{Z mu}
\mu_\alpha^n\vert\mathcal{L}_\xi Z_\alpha\vert_\xi(\mu_\alpha x)\to 0\quad\text{ and }\mathcal{C}^{0}_{loc}(\R^n)\quad\text{ as }\alpha\to\infty.
\end{equation}
We have denoted
\begin{equation}
x_0=\lim_{\alpha\to\infty}x_\alpha
\end{equation} 
and
\begin{equation}
U(x)=\left(1+\frac{f_0(x_0)}{n(n-2)}\vert x\vert^2\right)^{1-\frac{n}{2}}.
\end{equation}
\end{lemma}
\begin{proof}
The proof involves similar arguments to the ones used for Lemma \eqref{lem: first estimate}. Let $y_\alpha\in B_{x_\alpha}(8\rho_\alpha)$ be such that
\begin{equation}
\begin{array}{r l}
 u_\alpha^q(y_\alpha)+&\left\vert\frac{\nabla u_\alpha(y_\alpha)}{u_\alpha(y_\alpha)}\right\vert^n+\left\vert\frac{\nabla^2u_\alpha(y_\alpha)}{u_\alpha(y_\alpha)}\right\vert^\frac{n}{2} +\vert\mathcal{L}_g W_\alpha(y_\alpha)\vert \\
 & =\sup_{B_{x_\alpha}(8\rho_\alpha)}\left(u_\alpha^q(x)+\left\vert\frac{\nabla u_\alpha(x)}{u_\alpha(x)}\right\vert^n+\left\vert\frac{\nabla^2u_\alpha(x)}{u_\alpha(x)}\right\vert^\frac{n}{2}+\vert \mathcal{L}_gW_\alpha(x)\vert \right)
\end{array}
\end{equation}
and let
\begin{equation}
    \nu_\alpha^{-n}:=u_\alpha^q(y_\alpha)+\left\vert\frac{\nabla u_\alpha(y_\alpha)}{u_\alpha(y_\alpha)}\right\vert^{n}+\left\vert\frac{\nabla^2u_\alpha(y_\alpha)}{u_\alpha(y_\alpha)}\right\vert^\frac{n}{2}+\vert\mathcal{L}_gW_\alpha(y_\alpha)\vert.
\end{equation}
Conditions (\ref{est: blow up no 22}) and (\ref{est: blow up no 32}) imply that
\begin{equation}\label{nu and rho}
\frac{\rho_\alpha}{\nu_\alpha}\to \infty\quad\text{ and }\quad\nu_\alpha\to 0\quad\text{ as }\quad\alpha\to\infty.
\end{equation}
Moreover,
\begin{equation}
 d_g(x_\alpha,y_\alpha)\leq C_2^{\frac{1}{n}}\nu_\alpha,
\end{equation}
which implies that the coordinates of $y_\alpha$ in the exponential chart around $x_\alpha$, defined as $\tilde{y}_\alpha:=\nu^{-1}\exp^{-1}_{x_\alpha}(y_\alpha)$, are bounded by $C_2^{\frac{1}{n}}$. Up to a subsequence, we may choose a finite limit $\tilde{y}_0:=\lim_{\alpha\to\infty}\tilde{y}_\alpha$. We denote 
\begin{equation}
\hat{v}_\alpha(x)=\nu_\alpha^{\frac{n-2}{2}}u_\alpha\left(\exp_{x_\alpha}(\nu_\alpha x)\right)\quad\text{ and}\quad \hat{Z}_\alpha(x)=\nu_\alpha^{n-1}Z_\alpha\left(\exp_{x_\alpha}(\nu_\alpha x)\right)
\end{equation}
for $ x\in\Omega_\alpha:=B_0\left(\frac{8\rho_\alpha}{\nu_\alpha}\right).$ As before, 
\begin{equation}
\hat{v}_\alpha(x)=O(1),\quad \left|\frac{\nabla \hat{v}_\alpha(x)}{\hat{v}_\alpha(x)}\right|=O(1),\quad \left\vert\frac{\nabla^2\hat{v}_\alpha(x)}{\hat{v}_\alpha(x)}\right\vert=O(1)
\end{equation}
and
\begin{equation}\label{Z nu}
\vert\mathcal{L}_\xi\hat{Z}_\alpha(x)\vert_\xi\to 0.
\end{equation}
This implies that
\begin{equation}
\hat{v}_\alpha^q(\tilde{y}_\alpha)+\left\vert\frac{\nabla\hat{v}_\alpha(\tilde{y}_\alpha)}{\hat{v}_\alpha(\tilde{y}_\alpha)}\right\vert^n+\left\vert\frac{\nabla^2\hat{v}_\alpha(\tilde{y}_\alpha)}{\hat{v}_\alpha(\tilde{y}_\alpha)}\right\vert^\frac{n}{2}=1.
\end{equation}
By applying the same analysis as in the proof of Lemma $\ref{lem: first estimate}$, we get that, up to passing to a subsequence, there exists $U_\lambda:=\lim_{\alpha\to\infty}v_\alpha$ in $\mathcal{C}^{2,\eta}_{loc}(\R^n)$, with $x_0:=\lim_{\alpha\to\infty}x_\alpha$, such that 
\begin{equation}
\Delta U_{\lambda}= f(x_0) U_{\lambda}^{q-1},
\end{equation}
Since $\nabla U_\lambda(0)=0$, it holds that
\begin{equation}\label{U lambda}
U_\lambda(x)=\lambda^{\frac{n-2}{2}}\left(1+\frac{f(x_0)\lambda^2\vert x\vert^2}{n(n-2)}\right)^{1-\frac{n}{2}}
\end{equation}
for some $\lambda>0$, where
\begin{equation}
\frac{\nu_\alpha}{\mu_\alpha}\to \lambda.
\end{equation}
This yields \eqref{mu and rho}, \eqref{v mu} and \eqref{Z mu}, thanks to \eqref{nu and rho}, \eqref{U lambda} and \eqref{Z nu}.
\end{proof}

\subsection{The sphere of dominance around a blow-up point}
\par We denote
\begin{equation}
B_\alpha(x)=\mu_\alpha^{\frac{n-2}{2}}\left(\mu_\alpha^2+\frac{f_\alpha(x_\alpha)}{n(n-2)}\vert x\vert^2\right)^{1-\frac{n}{2}}
\end{equation}
and 
\begin{equation}
\theta_\alpha(x)=\sqrt{\mu_\alpha^2+\vert x\vert^2}.
\end{equation}
Our next goal is to extend the estimates from a ball of size $\mu_\alpha$ to one of size $\rho_\alpha$. We define the radius on which the estimates continue to hold as
\begin{equation}
r_\alpha=\sup \mathcal{R}_\alpha
\end{equation}
where 
\begin{equation}\label{influence}
\begin{array}{c}
\mathcal{R}_\alpha=\Big\{ 0<r\leq\rho_\alpha,\quad v_\alpha\leq (1+\varepsilon)B_\alpha,\quad\vert\nabla(v_\alpha-B_\alpha)\vert_{\xi}\leq \varepsilon\vert\nabla B_\alpha\vert_\xi, \\
\quad\quad\quad\quad\quad\text{ and }B_0(r)\backslash B_0(2R_\alpha\mu_\alpha)\Big\}
\end{array}
\end{equation}
where
\begin{equation}
R_\alpha^2=\frac{n(n-2)}{f_\alpha(x_\alpha)}.
\end{equation}
The two following properties hold for $r_\alpha$:
\begin{equation}\label{r first}
r_\alpha=O(\sqrt{\mu_\alpha})
\end{equation}
and
\begin{equation}\label{r second}
r_\alpha>>\mu_\alpha.
\end{equation}
By the previous lemma, we know that the $\mathcal{C}^{2,\eta}$ limit holds on balls of order $\rho_\alpha$ and  by definition also of size $r_\alpha$, which is to say that the two are comparable. As a result, \eqref{mu and rho} implies \eqref{r first}. In order to get the second estimate, it suffices to note that, by the definition of $r_\alpha$ and by \eqref{varepsilon2},
\begin{equation}
\varepsilon\leq C \mu_\alpha^{\frac{n-2}{2}}r_\alpha^{2-n},
\end{equation}
which directly implies \eqref{r second}.

\subsubsection{First order estimates of $v_\alpha$ on $B_0(8\delta_\alpha)$}
\begin{lemma}\label{approx on larger balls}
Let $(\delta_\alpha)_\alpha$ $0<\delta_\alpha\leq r_\alpha$ be a sequence of radii. Then for any $z_\alpha\in B_0(8\delta_\alpha)$ there holds:
\begin{equation}
 v_\alpha(z_\alpha)+\vert\nabla v_\alpha(z_\alpha)\vert \vert z_\alpha\vert+\vert\nabla^2 v_\alpha(z_\alpha)\vert\vert z_\alpha \vert^2\leq CB_\alpha(z_\alpha).
\end{equation}
Moreover, there exists a sequence of positive numbers $(\kappa_\alpha)_{\alpha\in\N}$ such that
\begin{equation}
(1-\kappa_\alpha)B_\alpha(z_\alpha)\leq v_\alpha(z_\alpha).
\end{equation}
If $\delta_\alpha\to 0$, then $\kappa_\alpha\to 0$.
\end{lemma}
\begin{proof}
For $x\in B_0(8)$:
\begin{equation}
\bar{v}_\alpha(x)=r_\alpha^{\frac{n-2}{2}}v_\alpha(r_\alpha x).
\end{equation}
Then $\bar{v}_\alpha$ satisfies 
\begin{equation}
\begin{array}{r l}
 \Delta_\xi \bar{v}_\alpha(x)+r_\alpha^2\bar{h}_\alpha(x) =& \bar{f}_\alpha(x)\bar{v}_\alpha^{q-1}(x)+r_\alpha^{2n}\frac{\bar{a}_\alpha(x)}{\bar{v}_\alpha^{q+1}(x)}-r_\alpha^n\frac{\tilde{b}_\alpha(x)}{\bar{v}_\alpha(x)}\\ 
& -r_\alpha^{\frac{n}{2}}\frac{\langle \nabla \bar{v}_\alpha(x),\bar{Y}_\alpha(x)\rangle}{\bar{v}_\alpha(x)}\left(\frac{r_\alpha^{\frac{n-2}{2}}\bar{d}_\alpha(x)}{\bar{v}_\alpha(x)}+\frac{r_\alpha^{\frac{3n-2}{n-2}}\bar{c}_\alpha(x)}{\bar{v}_\alpha^{q+1}(x)}\right)\\
& -r_\alpha^{2n+2}\frac{\langle\nabla\bar{v}_\alpha(x),\bar{Y}_\alpha(x) \rangle^2}{\bar{v}_\alpha(x)}\frac{1}{\bar{v}^{q+1}_\alpha(x)}
\end{array}
\end{equation}
where
\begin{equation}
\begin{array}{c}
\bar{a}_\alpha(x)=\tilde{a}_\alpha(r_\alpha x),\quad 
\bar{h}_\alpha(x)=\tilde{h}_\alpha(r_\alpha x),\quad
\bar{f}_\alpha(x)=\tilde{f}_\alpha(r_\alpha x),\\
\bar{c}_\alpha(x)=\tilde{c}_\alpha(r_\alpha x),\quad
\bar{d}_\alpha(x)=\tilde{d}_\alpha(r_\alpha x),\quad
\bar{Y}_\alpha(x)=\tilde{Y}_\alpha(r_\alpha x).
\end{array}
\end{equation}
By the definition of $r_\alpha$, we know that 
\begin{equation}
\bar{v}_\alpha(x)\leq C\left(\frac{\mu_\alpha}{r_\alpha}\right)^{\frac{n-2}{2}}
\end{equation}
in $B_0(1)\backslash B_0\left(\frac{1}{2}\right).$
By the weak estimate we know that
\smallskip
\begin{equation}
\begin{array}{r l}
\bar{v}_\alpha\leq C\\
\left\vert\frac{\nabla \bar v_\alpha(x)}{\bar v_\alpha(x)}\right\vert\leq C\\
\left \vert\frac{\nabla^2 \bar{v}_\alpha(x)}{\bar{v}_\alpha(x)} \right\vert\leq C \\
r_\alpha^{2n}\bar{a}_\alpha\leq C
\end{array}
\end{equation}
\smallskip
in $B_0(8)\backslash B_0\left(\frac{1}{2}\right)$. We conclude the proof by Lemma \ref{harnack lemma}. 
\par Considering $G_\alpha$ the Green function of $\Delta_g+h_\alpha$ in $M$. For any sequence $(z_\alpha)$ of points in $B_0(8r_\alpha)$:
\begin{equation}
v_\alpha(z_\alpha)\geq \varphi_\alpha(z_\alpha)\int_{B_0(r_\alpha)}\varphi_\alpha(y) G_\alpha\left(\Phi_\alpha^{-1}(z_\alpha),\Phi_\alpha^{-1}(y)\right)\title{f}_\alpha(y)v_\alpha^{q-1}(y)\,dy.
\end{equation}
In particular,
\begin{equation}
\begin{array}{r l}
 \frac{v_\alpha(z_\alpha)}{B_\alpha(z_\alpha)}\geq &\varphi_\alpha(z_\alpha)\int_{B_0(\frac{6r_\alpha}{\mu_\alpha})}\varphi_\alpha(\mu_\alpha y)\tilde{f}_\alpha(\mu_\alpha y)\left(\mu_\alpha^{\frac{n-2}{2}}v_\alpha(\mu_\alpha y)\right)^{q-1}  \\
 & \times G_\alpha(\Phi^{-1}_\alpha(z_\alpha),\Phi_\alpha^{-1}(\mu_\alpha y))d_g\left(\Phi^{-1}_\alpha(z_\alpha),\Phi_\alpha^{-1}(\mu_\alpha y)\right)^{n-2}\\
 &\times \left(\frac{\mu_\alpha^2+\frac{f_\alpha(x_\alpha)}{n(n-2)}\vert z_\alpha\vert^2}{d_g\left(\Phi_\alpha^{-1}(z_\alpha),\Phi_\alpha^{-1}(\mu_\alpha y)\right)^2}\right)^{\frac{n-2}{2}}\,dy.
\end{array}
\end{equation}
\end{proof}

\subsubsection{Improved weak estimate of $\mathcal{L}_\xi Z_\alpha$ on $B_0(7\delta_\alpha)$}

\begin{lemma}\label{improved weak estimate on LZ}
Let $(\delta_\alpha)_\alpha$ be a sequence of positive numbers such that $\delta_\alpha>>\mu_\alpha$ and $\delta_\alpha\leq \sqrt{\mu_\alpha}$. We get for any $x\in B_0(7\delta_\alpha)$,
\begin{equation}\label{LZ smaller than the ball}
\int_{B_0(6\delta_\alpha)}|x-y|^{2-n}|\mathcal{L}_\xi Z_\alpha(y)|^2_\xi v_\alpha^{-q-1}(y)\,dy\leq C\left( B_\alpha(x)+O(1)\right)
\end{equation}
and as a consequence
\begin{equation}\label{LZ ring}
\int_{B_0(6\delta_\alpha)\backslash B_0(\delta_\alpha)}|\mathcal{L}_\xi Z_\alpha|^2_\xi\,dy\leq C\left( \mu_\alpha^{2n-2}\delta_\alpha^{2-3n}+\mu_\alpha^{\frac{3n}{2}-1}\delta_\alpha^{-2n} \right),
\end{equation}
and there exists a sequence $s_\alpha\in (5\delta_\alpha, 6\delta_\alpha)$ such that
\begin{equation}\label{LZ bdry}
\int_{\partial B(s_\alpha)}|\mathcal{L}_\xi Z_\alpha|_\xi^2\,d\sigma\leq  C\left( \mu_\alpha^{2n-2}\delta_\alpha^{1-3n}+\mu_\alpha^{\frac{3n}{2}-1}\delta_\alpha^{-2n-1} \right).
\end{equation}
\end{lemma}
\begin{proof}
We use the Green's representation theorem for $\Delta_\xi +\tilde{h}_\alpha$ in $B_0(7\delta_\alpha)$ in the $1$st equation, and obtain
\begin{equation}
\int_{B_0(6\delta_\alpha)}\vert x-y\vert^{2-n}\frac{\vert\mathcal{L}_\xi Z_\alpha(y)\vert^2_\xi}{v_\alpha^{q+1}(y)}\, dy\leq C\left(B_\alpha(x)+H_1+H_2+H_3\right),
\end{equation}
where
\begin{equation}
\begin{array}{r l}
 H_1=&\int_{B_0(6\delta_\alpha)} \frac{\tilde{b}_\alpha(y)}{v_\alpha(y)}\vert x-y \vert^{2-n}\, dy, \\
 H_2=&\int_{B_0(6\delta_\alpha)} \langle\nabla v_\alpha(y),\tilde{V}(y)\rangle\left(\frac{\tilde{d}_\alpha(y)}{v_\alpha^2(y)}+\frac{\tilde{c}_\alpha(y)}{v_\alpha^{q+2}(y)}\right)\vert x-y\vert^{2-n}\, dy, \\
 H_3=&\int_{B_0(6\delta_\alpha)} \frac{\langle \nabla v_\alpha(y),\tilde{V}(y)\rangle^2}{v_\alpha^{q+3}(y)} \vert x-y\vert^{2-n}\, dy. \\
\end{array}
\end{equation}
Lemma \ref{approx on larger balls} yields the following estimates:
\begin{equation}
H_1\leq C\int_{B_0(6\delta_\alpha)}\mu_\alpha^{\frac{2-n}{n}}\theta_\alpha^{n-2}(y)\vert x-y\vert^{2-n}\,dy\leq C \mu_\alpha\left(\frac{\delta_\alpha^2}{\mu_\alpha} \right)^\frac{n}{2},
\end{equation}
\begin{equation}
\begin{array}{c}
H_2\leq C\int_{B_0(6\delta_\alpha)}\theta_\alpha^{-2}(y)\left(\mu_\alpha^{\frac{2-3n}{3}} \theta_\alpha^{3n-2}(y)+\mu_\alpha^{\frac{2-n}{2}} \theta_\alpha^{n-2}(y) \right)\quad\quad\\ 
\quad\quad\quad\quad\quad\quad\quad\quad\quad\quad\quad\quad\quad\quad\quad\quad\quad\quad\quad\quad\quad\quad \times\vert x-y\vert^{2-n}\,dx,\\
\leq C\left(\frac{\delta_\alpha^2}{\mu_\alpha}\right)^\frac{3n-2}{2}+C\left(\frac{\delta_\alpha^2}{\mu_\alpha}\right)^\frac{n-2}{2}\\
\end{array}
\end{equation}
and
\begin{equation}
H_3\leq C\int_{B_0(6\delta_\alpha)}\mu_\alpha^{\frac{2-3n}{2}}\theta_\alpha^{3n-2}(y)\vert x-y\vert^{2-n}\,dx\leq C\left(\frac{\delta_\alpha^2}{\mu_\alpha} \right)^\frac{3n-2}{2}.
\end{equation}
As a consequence,
\begin{equation}
 \int_{B_0(6\delta_\alpha)} \vert x-y\vert^{2-n}\frac{\vert\mathcal{L}_\xi Z_\alpha(y)\vert^2_\xi}{v_\alpha^{q+1}(y)}\, dy\leq C\left( \mu_\alpha^{\frac{n-2}{2}}\theta_\alpha^{2-n}(x)+1\right).
\end{equation}
In particular, we get \eqref{LZ ring} and \eqref{LZ bdry}.
\end{proof}

\subsubsection{First order estimate of $\mathcal{L}_\xi Z_\alpha$ on $B_0(3\delta_\alpha)$}
We use the previous improved weak estimate in order to get a first order estimate of $\mathcal{L}_\xi Z_\alpha$.
For $x\not =0$, let
\begin{equation}
\mathcal{G}_i(x)_j=-\frac{1}{4(n-1)\omega_{n-1}}\vert x\vert ^{2-n}\left((3n-2)\delta_{ij}+(n-2)\frac{y_i y_j}{\vert x\vert ^2}\right)
\end{equation}
be the $i$-th fundamental solution of $\overrightarrow{\Delta}_\xi$ in $\R^n$. 
We define on $\R^n$ the vector field 
\begin{equation}
V_\alpha(x)_i=-\frac{n^2}{2(n-2)}\ln\left(1+\frac{\vert x\vert^2}{\mu_\alpha^2}\right)\tilde{V}_\alpha(0)_i+\frac{n}{\mu_\alpha^2+\vert x\vert^2}\left\langle x,\tilde{V}_\alpha(0)\right\rangle x_i
\end{equation}
and a vector field $R_\alpha$ such that
\begin{equation}
\begin{array}{c}
\overrightarrow{\Delta}_\xi(V_\alpha+R_\alpha)(x)=2\frac{n-1}{n-2}\Big(-\frac{\left\langle\nabla^2 B_\alpha(x),\tilde{V}_\alpha(0)\right\rangle}{B_\alpha(x)}\quad\quad\quad\\
\quad\quad\quad\quad\quad\quad\quad\quad\quad\quad\quad\quad\quad\quad\quad\quad\quad+\frac{3n-2}{n-2}\frac{\left\langle\nabla B_\alpha(x),\tilde{V}_\alpha(0)\right\rangle\nabla B_\alpha(x)}{B^2_\alpha(x)}\Big).
\end{array}
\end{equation}
Note, in particular, that
\begin{equation}
\vert \overrightarrow{\Delta}_\xi R_\alpha(x) \vert\leq C\vert\tilde{V}_\alpha(0)\vert \mu_\alpha^2\theta_\alpha(x)^{-4}.
\end{equation}
Thus, 
\begin{equation}
\begin{array}{r| l c}
& C\vert\tilde{V}_\alpha(0)\vert \mu_\alpha^2\theta_\alpha^{-3}(x) & n= 5\\ \\
 \vert\mathcal{L}_\xi R_\alpha(x) \vert\leq 
&  C\vert\tilde{V}_\alpha(0)\vert \mu_\alpha^2\theta_\alpha^{-3}(x)\ln\left(1+ \frac{\theta_\alpha(x)}{\mu_\alpha}\right)& n=4, \\ \\
&  C\vert\tilde{V}_\alpha(0)\vert \mu_\alpha^2\theta_\alpha^{-2}(x) & n=3.
\end{array}
\end{equation}
By direct calculation, we see that
\begin{equation}
\begin{array}{r l}
 \left(\mathcal{L}_\xi V_\alpha\right)_{ij}(x)=& -\frac{2n}{n-2}\vert\tilde{V}_\alpha(0)\vert\frac{\vert x\vert}{\mu_\alpha^2+\vert x\vert^2}\left(\frac{x_j}{\vert x\vert}\frac{\tilde{V}_\alpha(0)_i}{\vert \tilde{V}_\alpha(0)\vert}+\frac{x_i}{\vert x\vert}\frac{\tilde{V}_\alpha(0)_j}{\vert \tilde{V}_\alpha(0)\vert}-\frac{2}{n}\left\langle\frac{x}{\vert x\vert},\frac{\tilde{V}_\alpha(0)}{\vert\tilde{V}_\alpha(0)\vert}\right\rangle\right)\\
 & -4\frac{\vert x\vert^3\vert \tilde{V}_\alpha(0)\vert}{\left(\mu_\alpha^2+\vert x\vert^2\right)^2}\left\langle\frac{x}{\vert x\vert},\frac{\tilde{V}_\alpha(0)}{\vert\tilde{V}_\alpha(0)\vert}\right\rangle\\
 &-4n\frac{\vert x\vert^3\vert \tilde{V}_\alpha(0)\vert}{\left(\mu_\alpha^2+\vert x\vert^2\right)^2}\left\langle\frac{x}{\vert x\vert},\frac{\tilde{V}_\alpha(0)}{\vert\tilde{V}_\alpha(0)\vert}\right\rangle\frac{x_i}{\vert x\vert}\frac{x_j}{\vert  x\vert}.
\end{array}
\end{equation}
Note that
\begin{equation}\label{V bdry}
\left\vert \mathcal{L}_\xi V_\alpha(x) \right\vert\leq C\vert\tilde{V}_\alpha(0)\vert\theta_\alpha^{-1}(x).
\end{equation}
\begin{lemma} Let $(\delta_\alpha)_\alpha$ be a sequence of positive numbers such that 
\begin{equation}
\frac{\mu_\alpha}{\delta_\alpha}\to 0\quad\text{ and }\quad \delta_\alpha\leq\min(r_\alpha,\sqrt{\mu_\alpha}).
\end{equation}
For any $x\in B_0(3\delta_\alpha)$, we get the following estimate on $\vert\mathcal{L}_\xi\left(Z_\alpha-V_\alpha\right)(x)\vert$:
\begin{equation}\label{new est Z}
\vert\mathcal{L}_\xi Z_\alpha(x)\vert\leq \theta_\alpha^{-1}(x)+\mu_\alpha^{n-1}\delta_\alpha^{1-2n}.
\end{equation}
\end{lemma}
\begin{proof}
Without making mention of the conformal change factor $\varphi_\alpha$,
\medskip
We apply the Green representation theorem on the $2$nd equation. Let $\mathcal{G}_{\alpha,i}$ be the $i$-th Green $1$-form for $\overrightarrow{\Delta}_\xi$ with Neumann boundary conditions on $B_0(s_\alpha)$, $s_\alpha\leq 4\delta_\alpha$. Similarly, let
\begin{equation}
\mathcal{H}_{ij,\alpha}(x,y)_p=\partial_i\mathcal{G}_{\alpha,j}(x,y)_p+\partial_j\mathcal{G}_{\alpha,i}(x,y)_p-\frac{2}{n}\xi_{ij}\sum_{k=1}^n\partial_k\mathcal{G}_{\alpha,k}(x,y)_p.
\end{equation}
There holds that
\begin{equation}
\begin{array}{r l}
\mathcal{L}_\xi(Z_\alpha-V_{\alpha}-R_\alpha)_{ij}(x)=&\int_{B_0(s_\alpha)}\mathcal{H}_{ij,\alpha}(x,y)_p\overrightarrow{\Delta}_\xi (Z_\alpha-V_{\alpha}-R_\alpha)^p(y)\,dy\\
&+\int_{\partial B_0(s_\alpha)}\mathcal{H}_{ij,\alpha}(z_\alpha,y)_p\nu_p\mathcal{L}_\xi (Z_\alpha-V_{\alpha}-R_\alpha)^{pq}(y)\,d\sigma.
\end{array}
\end{equation}
Keeping in mind that $R_\alpha$ is negligible compared to $V_\alpha$, we obtain the estimate
\begin{equation}\label{estimate-Zalpha}
\vert\mathcal{L}_\xi(Z_\alpha-V_\alpha-R_\alpha)(x) \vert_\xi\leq C\left(I_1+I_2+I_3+I_4+J_1+J_2\right),
\end{equation}
where the bulk terms are
\begin{equation}
\begin{array}{r l}
I_1 =&  \int_{B_0(6\delta_\alpha)}\vert x-y \vert^{1-n}\vert\mathcal{L}_\xi Z_\alpha(y) \vert\, dy\\
I_2 =& \Big\vert 2\frac{n-1}{n-2}\int_{B_0(6\delta_\alpha)}\frac{3n-2}{n-2}\left(\frac{\left\langle\nabla v_\alpha(y),\tilde{V}_\alpha(y)\right\rangle\nabla v_\alpha(y)}{v_\alpha^2(y)}- \frac{\left\langle\nabla B_\alpha(y),\tilde{V}_\alpha(y)\right\rangle\nabla B_\alpha(y)}{B_\alpha^2(y)}\right)\vert x-y \vert^{1-n}\\
& \quad\quad\quad\quad -\left(\frac{\left\langle\nabla^2 v_\alpha(y),\tilde{V}_\alpha(y)\right\rangle}{v_\alpha(y)}- \frac{\left\langle\nabla^2 B_\alpha(y),\tilde{V}_\alpha(y)\right\rangle}{B_\alpha(y)} \right)\vert x-y \vert^{1-n}\,dy\Big\vert\\
I_3=& \Big\vert\int_{B_0(6\delta_\alpha)} \vert x-y\vert^{1-n}\Big(-\frac{\left\langle\nabla v_\alpha(y),\tilde{V}_\alpha(y)\right\rangle}{v_\alpha(y)}\nabla\ln\tilde{N}_\alpha(y)\\
&\quad\quad\quad\quad\quad\quad\quad+div\tilde{V}_\alpha(y)\frac{\nabla v_\alpha(y)}{v_\alpha(y)}-\frac{\langle\nabla\tilde{V}(y),\nabla v_\alpha(y) \rangle}{v_\alpha(y)}\Big)\,dy\Big\vert\\
I_4=& \int_{B_0(6\delta_\alpha)} \vert x-y\vert^{1-n}\, dy\\
\end{array}
\end{equation}
\smallskip
\smallskip and the boundary terms
\begin{equation}
\begin{array}{r l}
J_1=&\int_{\partial B_0}\vert x-y\vert^{1-n}\vert\mathcal{L}_\xi V_\alpha(x)\vert\,d\sigma,\\
J_2=&\int_{\partial B_0}\vert x-y\vert^{1-n}\vert\mathcal{L}_\xi Z_\alpha(x)\vert\,d\sigma.
\end{array}
\end{equation}
Then, by \eqref{V bdry}, 
\begin{equation}
J_1\leq C \delta_\alpha^{-1},
\end{equation}
and by \eqref{LZ bdry},
\begin{equation}
J_2\leq\mu_\alpha^{n-1} \delta_\alpha^{1-2n}.
\end{equation}
Next, we see that
\begin{equation}
\begin{array}{r l}
I_2\leq & C\int_{B_0(6\delta_\alpha)}\vert x-y\vert^{1-n}\theta_\alpha^{-2}(y)\,dy \\
 \leq & C\theta_\alpha^{-1}(x)\int_{B_0\left(\frac{6\delta_\alpha}{\theta_\alpha(x)} \right)}\left\vert\frac{x}{\theta_\alpha(x)}-z \right\vert^{1-n}\frac{1}{\left(\frac{\mu_\alpha}{\theta_\alpha(x)}\right)^2+\left\vert z\right\vert^2}\,dz
\end{array}
\end{equation}
so that
\begin{equation}
\left\vert I_2\right\vert\leq  \theta_\alpha^{-1}(x).
\end{equation}
The term $I_3$ is in fact negligible when compared to $I_2$ and so 
\begin{equation}\left\vert I_3\right\vert\leq \theta_\alpha^{-1}(x)
\end{equation}
also.
It is also clear that
\begin{equation}
I_4\leq C\delta_\alpha.
\end{equation}
Coming back to \eqref{estimate-Zalpha} with all these estimates, we thus obtain that 
\begin{equation}\label{estimate-Zalpha1}
\left\vert {\mathcal L}_\xi Z_\alpha(x)\right\vert \le C\left(\theta_\alpha(x)^{-1}+\mu_\alpha^{n-1}\delta_\alpha^{1-2n}\right) + I_1.
\end{equation}
It remains to estimate $I_1$. We shall use an iterative argument to do it. Assume that 
\begin{equation}\label{estimate-Zalpha2}
\left\vert {\mathcal L}_\xi Z_\alpha(x)\right\vert \le C\left(\theta_\alpha(x)^{-\beta}+\mu_\alpha^{n-1}\delta_\alpha^{1-2n}\right)
\end{equation}
for some $1<\beta\le n$. Note that, thanks to the weak estimate \eqref{weak} on ${\mathcal L}_\xi Z_\alpha$, it holds for $\beta=n$. If \eqref{estimate-Zalpha2} holds, we can write that 
\begin{equation}
\begin{array}{rcl}
I_1&\le & C \mu_\alpha^{n-1}\delta_\alpha^{1-2n}\int_{B_0(6\delta_\alpha)}\vert x-y \vert^{1-n} \, dy\\
&&+ C \int_{B_0(6\delta_\alpha)}\vert x-y \vert^{1-n}\theta_\alpha(y)^{-\beta} \, dy\\
&\le& C \mu_\alpha^{n-1}\delta_\alpha^{2-2n}\\ 
&& +\left\{\begin{array}{ll}
\theta_\alpha(x)^{1-\beta}&\hbox{ if }\beta<n\\
\theta_\alpha(x)^{1-n}\ln\left(1+\frac{\theta_\alpha(x)}{\mu_\alpha}\right)&\hbox{ if }\beta=n\\
\end{array}\right.
\end{array}
\end{equation}
Remember here that $\beta>1$. Coming back to \eqref{estimate-Zalpha1}, we obtain that, if \eqref{estimate-Zalpha2} holds for some $1<\beta\le n$, it necessarily also holds when $\beta$ is replaced by $\beta-\frac{1}{2}$. Since, as already said, it holds for $\beta=n$, we obtain by induction that it holds for all $\beta=n-\frac{k}{2}$ as long as $n-\frac{k-1}{2}>1$. Thus, it holds for $\beta=1$. But this is exactly the estimate \eqref{new est Z}.
\end{proof}

\begin{rem}
For $\delta_\alpha=r_\alpha$, we get that
\begin{equation}
\vert\mathcal{L}_\xi Z_\alpha\vert\leq\theta_\alpha^{-1}+\mu_\alpha^{n-1}\delta_\alpha^{1-2n}
\end{equation}
implies
\begin{equation}
\vert \mathcal{L}_\xi Z_\alpha\vert\leq \left(\frac{\mu_\alpha}{r_\alpha} \right)^{n-1}\theta_\alpha^{-n}.
\end{equation}
\end{rem}
\subsubsection{Asymptotic profile on $B_0(2r_\alpha)$}
\begin{lemma}\label{profile on Br}
Up to a subsequence, it holds that
\begin{equation}\label{asymptotic profile}
v_\alpha(0)r_\alpha^{n-2}v_\alpha(r_\alpha x)\to\frac{R_0^{n-2}}{\vert x\vert^{n-2}}+H(x)\quad\text{ in }\mathcal{C}^2_{loc}(B_0(2)\backslash\{0\}),
\end{equation}
where $H$ is a non-negative superharmonic function in $B_0(2)$. We recall that, by \eqref{mu v},
\begin{equation}
v_\alpha(0)=\mu_\alpha^{1-\frac{n}{2}}.
\end{equation}
\end{lemma}
\begin{proof}
\textbf{Step 1:} Let
\begin{equation}
\check{v}_\alpha(x)=\mu_\alpha^{1-\frac{n}{2}}r_\alpha^{n-2}v_\alpha(r_\alpha x),\quad x\in B_0(2),
\end{equation}
where $\mu_\alpha$ is defined in \eqref{mu v}. Then
\begin{equation}
\Delta_\xi \check{v}_\alpha=\check{F}_\alpha,
\end{equation}
with 
\begin{equation}\label{F tilde}
\begin{array}{r l}
     \check{F}_\alpha=&     -r_\alpha^{2}\check{h}_\alpha(x)\check{v}_\alpha(x) +\mu_\alpha^{2}r_\alpha^{-2}\check{f}_\alpha (x)\check{v}_\alpha^{q-1}(x) \\ 
    &     +\mu_\alpha^{2-2n}r_\alpha^{4n-2}\frac{\check{\rho}_{1,\alpha}(x)+\vert\check{\Psi}_\alpha(x)+\check{\rho}_{2,\alpha}(x)\mathcal{L}_\xi Z_\alpha(x)\vert^2_\xi}{\check{v}^{q+1}_\alpha(x)}\\ &-\mu_\alpha^{2-n}r_\alpha^{2n-2}\frac{\check{b}_\alpha(x)}{\check{v}_\alpha (x)}-\mu_\alpha^{\frac{2-n}{2}}r_\alpha^{n-2} \frac{\langle\nabla \check{v}_\alpha(x), \check{Y}_\alpha(x)\rangle^2}{\check{v}_\alpha^2(x)u_\alpha^{q+1}\left(\exp_{x_\alpha}(r_\alpha x)\right)}\\ 
    &  -\mu_\alpha^{1-\frac{n}{2}}r_\alpha^{n-1} \frac{\langle\nabla \check{v}_\alpha(x),\check{Y}_\alpha(x)\rangle}{\check{v}_\alpha(x)} \left( \frac{\check{d}_\alpha(x)}{v_\alpha(x)}+\frac{\check{c}_\alpha(x)}{v_\alpha(x)^{q+1}}\right)
\end{array}
\end{equation}
This implies that
\begin{equation}
\begin{array}{c}
\Delta_{\check{g}_\alpha}\check{v}_\alpha= \left(\frac{\mu_\alpha}{r_\alpha}\right)^2\check{f}_\alpha\check{v}_\alpha^{q-1}+\left(\frac{r_\alpha^2}{\mu_\alpha} \right)^{2n}\left(\frac{\mu_\alpha}{r_\alpha}\right)^2\frac{\vert\check{\rho}_{2,\alpha}\vert^2\vert\mathcal{L}_\xi Z_\alpha(r_\alpha x) \vert^2}{\check{v}_\alpha^{q+1}}\\
-\left(\frac{r_\alpha^2}{\mu_\alpha}\right)^{2n-2}\frac{\langle\nabla \check{v}_\alpha,\check{Y}_\alpha \rangle^2}{\check{v}_\alpha^{q+3}}+o(1);
\end{array}
\end{equation}
where
\begin{equation}
\check{f}_\alpha(x)=\tilde{f}_\alpha(r_\alpha x), \quad\check{\rho}_{2,\alpha}=\tilde{\rho}_{2,\alpha}(r_\alpha x),\text{ and }\quad \check{Y}_\alpha(x)=\tilde{Y}_\alpha(r_\alpha x).
\end{equation}
By the definition \eqref{influence}, there holds for some positive $C$ that
\begin{equation}\label{check two}
\left(\left(\frac{\mu_\alpha}{r_\alpha}\right)^{2}+\frac{f_\alpha(x_\alpha)}{n(n-2)}\vert x\vert^2\right)^{1-\frac{n}{2}}\leq\check{v}_\alpha(x)\leq \left(\frac{f_\alpha(x_\alpha)\vert x\vert^2}{n(n-2)}\right)^{1-\frac{n}{2}}.
\end{equation}
Similarly,
\begin{equation}
\vert \nabla\check{v}_\alpha(x)\vert\leq \left(\frac{f_\alpha(x_\alpha)\vert x\vert^2}{n(n-2)}\right)^{-\frac{n}{2}}.
\end{equation}
Moreover, for any $x\in B_0(2)$:
\begin{equation}\label{check three}
\frac{\check{a}_\alpha(x)}{\check{v}_\alpha^{q+1}(x)}\leq C\left(\left(\frac{\mu_\alpha}{r_\alpha}\right)^{2}+\frac{f_\alpha(x_\alpha)}{n(n-2)}\vert x\vert^2\right)^{\frac{n}{2}}\in L^{\infty}(B_0(2)\backslash\{0\}).
\end{equation}
We recall that we've assumed $\check{Y}_\alpha\to 0$ in $\mathcal{C}^{0,\alpha}$. By standard elliptic theory, we see that
\begin{equation}
\check{v}_\alpha\to\check{v}\quad\text{ in }\mathcal{C}^1_{loc}(B_0(2)\setminus\{0\})\quad\text{ as }\alpha\to\infty.
\end{equation}
For $x\not =0$,
\begin{equation}
\check{v}(x)=\frac{\lambda_0}{\vert x\vert^{n-2}}+H(x),
\end{equation}
where $H$ is a superharmonic function in $B_0(2)$ and $\lambda_0=\left(\frac{n(n-2)}{f(x_0)}\right)^{1-\frac{n}{2}}$. Moreover, $H\geq 0$ in $B_0(2)$. If $r_\alpha<\rho_\alpha$, then $H(0)>0$. Indeed, by the definition \eqref{influence}, there exists $y_\alpha\in B_0(r_\alpha)$ such that at least one of the following conditions hold: 
\begin{enumerate}
\item $v_\alpha(y_\alpha)=(1+\varepsilon)B_\alpha(y_\alpha)$,
\item $\vert\nabla v_\alpha(y_\alpha)\vert_\xi=(1+\varepsilon)\vert\nabla B_\alpha(y_\alpha)\vert_\xi$,
\end{enumerate}
\par Letting $\check{y}_\alpha=\frac{y_\alpha}{r_\alpha},$ we see that either $H(\check{y}_\alpha)$ or $\nabla H(\check{y}_\alpha)$ are non-zero, and since $H$ is a non-negative superharmonic function, then $H(0)>0.$ Independently, we show that $H(0)\leq 0$. The Pohozaev identity writes as
\begin{equation}
\begin{array}{c}
\int_{B_0(\delta r_\alpha)}\left(x^k\partial_k v_\alpha(x)+\frac{n-2}{2}v_\alpha(x)\right)\Delta_\xi v_\alpha(x)\,dx\quad\quad\quad\quad \quad\quad\quad\quad \quad\quad\quad\quad\\
\quad\quad\quad\quad =\int_{\partial B_0(\delta r_\alpha)}\left( \frac{1}{2}\delta r_\alpha\vert\nabla v_\alpha(x) \vert^2_\xi-\frac{n-2}{2}v_\alpha(x)\partial_\nu v_\alpha(x) -\delta r_\alpha(\partial_\nu v_\alpha(x))^2 \right)\,d\sigma.
\end{array}
\end{equation}
Thanks to \eqref{asymptotic profile}, we can estimate the boundary terms as 
\begin{equation}
\begin{array}{l}
\int_{\partial B_0(\delta r_\alpha)}\left( \frac{1}{2}\delta r_\alpha\vert\nabla v_\alpha(x) \vert^2_\xi-\frac{n-2}{2}v_\alpha(x)\partial_\nu v_\alpha(x) -\delta r_\alpha(\partial_\nu v_\alpha(x))^2 \right)\,d\sigma \\
\quad = \left(\frac{\mu_\alpha}{r_\alpha}\right)^{n-2} \left(\int_{\partial B_0(\delta)}\left(\frac{1}{2}\delta \left\vert \nabla\Psi\right\vert^2-\frac{n-2}{2}\Psi\partial_\nu \Psi-\delta \left(\partial_\nu\Psi\right)^2\right)d\sigma+o(1)\right)
\end{array}
\end{equation}
where $\Psi(x)=R_0^{n-2}\left\vert x\right\vert^{2-n}+H(x)$. Simple computations lead then to 
\begin{equation}
\begin{array}{l}
\int_{\partial B_0(\delta r_\alpha)}\left( \frac{1}{2}\delta r_\alpha\vert\nabla v_\alpha(x) \vert^2_\xi-\frac{n-2}{2}v_\alpha(x)\partial_\nu v_\alpha(x) -\delta r_\alpha(\partial_\nu v_\alpha(x))^2 \right)\,d\sigma \\
\quad = \left(\frac{\mu_\alpha}{r_\alpha}\right)^{n-2}\left(\frac{(n-2)^2}{2}\omega_{n-1}R_0^{n-2}H(0)+O\left(\delta\right)\right).
\end{array}
\end{equation}
On the other hand, the LHS writes as
\begin{equation}
\int_{B_0(\delta r_\alpha)}\left(x^k\partial_k v_\alpha(x)+\frac{n-2}{2}v_\alpha(x) \right)\Delta_\xi v_\alpha(x)\,dx=J_1+J_2+J_3+J_4,
\end{equation}
where
\begin{equation}
\begin{array}{r l}
 J_1=&-\int_{B_0(\delta r_\alpha)}\left(x^k\partial_k v_\alpha(x)+\frac{n-2}{2}v_\alpha(x)\right)\\ 
 &\quad\quad\quad\quad\quad \times\left(\tilde{h}_\alpha(x) v_\alpha(x)+\frac{\tilde{b}_\alpha(x)}{v_\alpha(x)}+\langle\nabla v_\alpha(x),\tilde{Y}_\alpha(x) \rangle\left(\frac{\tilde{d}_\alpha(x)}{v_\alpha^2(x)}+\frac{\tilde{c}_\alpha(x)}{v_\alpha^{q+2}(x)} \right) \right)\,dx  \\
 J_2=&\int_{B_0(\delta r_\alpha)}\left(x^k\partial_k v_\alpha(x)+\frac{n-2}{2}v_\alpha(x)\right)\tilde{f}_\alpha(x) v_\alpha^{q-1}(x)\,dx  \\
 J_3=&\int_{B_0(\delta r_\alpha)}\left(x^k\partial_k v_\alpha(x)+\frac{n-2}{2}v_\alpha(x)\right)\frac{\tilde{\rho}_{1,\alpha}(x)+\vert\tilde{\Psi}_\alpha(x)+\tilde{\rho}_{2,\alpha}(x)\mathcal{L}_\xi Z_\alpha(x)\vert^2}{v_\alpha^{q+1}(x)}\, dx  \\
 J_4=&-\int_{B_0(\delta r_\alpha)}\left(x^k\partial_k v_\alpha(x)+\frac{n-2}{2}v_\alpha(x)\right)\frac{\langle \nabla v_\alpha(x), \tilde{Y}_\alpha(x) \rangle^2}{v_\alpha^{q+3}(x)}\, dx  \\
\end{array}
\end{equation}
We find estimates for each quantity in turn. In the case of $J_1$, we notice that
\begin{equation}
\left\vert\int_{B_0(\delta r_\alpha)}\tilde{h}_\alpha(x) B_\alpha^2(x)\, dx\right\vert\leq  C \left\{
\begin{array}{ll}
\mu_\alpha^2&\hbox{ if }n=5\\
\mu_\alpha^2 \ln\left(\frac{r_\alpha}{\mu_\alpha}\right)&\hbox{ if }n=4\\
\delta r_\alpha\mu_\alpha&\hbox{ if }n=3\\
\end{array}\right. 
\end{equation}
Then we have that 
\begin{equation}
 \left\vert \int_{B_0(\delta r_\alpha)} \tilde{b}_\alpha(x)\, dx\right\vert\leq C(\delta r_\alpha)^n,
\end{equation}
and
\begin{equation}
\begin{array}{r l}
\left\vert\int_{B_0(\delta r_\alpha)}  \frac{\langle \nabla B_\alpha(x), \tilde{Y}_\alpha(x) \rangle}{B_\alpha(x)}\left(\tilde{d}_\alpha(x)+\frac{\tilde{c}_\alpha(x)}{v_\alpha^q(x)}\right)\, dx\right\vert \leq & C\int_{B_0(\delta r_\alpha)}\theta_\alpha^{-1}(x)\,dx \\
\leq  &  C(\delta r_\alpha)^{n-1}.
\end{array}
\end{equation}
For $J_3$, we obtain
\begin{equation}
\begin{array}{r l}
\left\vert \int_{B_0(\delta r_\alpha)}\frac{\vert\mathcal{L}_\xi Z_\alpha(x) \vert^2}{B_\alpha^q(x)}\,dx\right\vert \leq& \int_{B_0(\delta r_\alpha)}\left(\frac{\mu_\alpha}{r_\alpha} \right)^{2n-2}\mu_\alpha^{-n}\theta_\alpha^n(x)  \,dx \\
\leq & C\left(\frac{\mu_\alpha}{r_\alpha} \right)^{n-2}r_\alpha\delta^{2n},
\end{array}
\end{equation}
while for $J_4$, we get
\begin{equation}
\begin{array}{r l}
\left\vert \int_{B_0(\delta r_\alpha)}\frac{\langle\nabla B_\alpha(x), \tilde{Y}_\alpha(x) \rangle^2}{B_\alpha^{q+3}(x)}\,dx\right\vert \leq& \int_{B_0(\delta_\alpha r_\alpha)}\theta_\alpha^{2n-2}(x)\mu_\alpha^{-n}\,dx\\
\leq & C\left(\frac{\mu_\alpha}{r_\alpha} \right)^{n}\left(\frac{r_\alpha^2}{\mu_\alpha} \right)^{2n}r_\alpha^{-2}\delta^{3n-2}.
\end{array}
\end{equation}
For $J_2$, lengthy, yet straightforward computations as those seen in \cite{Val19} lead to
\begin{equation}
J_2=o\left(\frac{\mu_\alpha}{r_\alpha} \right)^{n-2}.
\end{equation}
\par We conclude that
\begin{equation}
H(0)=o\left(\frac{\mu_\alpha}{r_\alpha} \right)^{n-2}(1+O(\delta)),\quad \forall \alpha,\quad \forall \delta>0,
\end{equation}
and thus $H(0)=0$.
\smallskip

\subsection{Stability theorem proof}
Consider the sets $\mathcal{S}_\alpha$ and let
\begin{equation}
16\delta_\alpha:=\min_{1\leq i<j\leq N_\alpha}\vert x_{i,\alpha}- x_{j,\alpha}\vert.
\end{equation}
We first prove that $\delta_\alpha\to 0$ as $\alpha\to +\infty$. Assuming that the contrary holds, we can apply the results of Lemma \ref{first blow up} with $x_\alpha=x_{1,\alpha}$ and $\rho_\alpha=\delta$ for some $\delta>0$ fixed. This contradicts \eqref{influence}. We reorder the elements of the sets $S_\alpha$ in order of distance, so that
\begin{equation}
16\delta_\alpha=\vert x_{1,\alpha}- x_{2,\alpha}\vert.
\end{equation}
 For $R>1$, let $1\leq M_{R,\alpha}$ be such that
\begin{equation}
\begin{array}{c}
\vert x_{1,\alpha}- x_{i_\alpha,\alpha}\vert \leq R\delta_\alpha\quad\text{ for }\quad i_\alpha\in\{1,\dots,M_{R,\alpha}\},\\
\vert x_{1,\alpha}- x_{i_\alpha,\alpha}\vert > R\delta_\alpha\quad\text{ for }\quad i_\alpha\in\{M_{R,\alpha}+1,\dots,N_\alpha\}.
\end{array}
\end{equation}
For $x\in B_0(8\delta_\alpha)$, we define the rescaled quantities
\begin{equation}
\check{v}_\alpha(x):=\delta_\alpha^{\frac{n-2}{2}}\varphi_\alpha(\delta_\alpha x)u_\alpha\circ \Phi^{-1}_\alpha(\delta_\alpha x)
\end{equation}
and 
\begin{equation}
\check{Z}_\alpha(x)=\delta_\alpha^{n-1}\varphi_\alpha^{-q+2}(\delta_\alpha x)(\Phi_\alpha)_* W_\alpha(\delta_\alpha x).
\end{equation}
In the exponential chart, the elements of $\mathcal{S}_\alpha$ become
\begin{equation}
\check{x}_{i,\alpha}:=\delta_\alpha^{-1}\exp_{x_{1,\alpha}}^{-1}(x_{i,\alpha}),
\end{equation}
where $1\leq i\leq N_i$. Note that $B_{x_{i,\alpha}}\left(8\delta_\alpha\right)$ and $B_{x_{j,\alpha}}\left(8\delta_\alpha\right)$ are disjoint. We define two types of concentration points : the first
\begin{equation}
\sup_{B_{\check{x}_{i,\alpha}}\left(8\right)}\left(\check{v}_\alpha(x)^q+\left\vert\frac{\nabla \check{v}_\alpha(x)}{\check{v}_\alpha(x)} \right\vert^n+\left\vert\frac{\nabla^2 \check{v}_\alpha(x)}{\check{v}_\alpha(x)} \right\vert^\frac{n}{2}+\vert \mathcal{L}_\xi \check{Z}_\alpha(x)\vert \right)=O(1)
\end{equation}
and the second
\begin{equation}
\sup_{B_{\check{x}_{i,\alpha}}\left(8\right)}\left(\check{v}_\alpha(x)^q+\left\vert\frac{\nabla \check{v}_\alpha(x)}{\check{v}_\alpha(x)} \right\vert^n+\left\vert\frac{\nabla^2 \check{v}_\alpha(x)}{\check{v}_\alpha(x)} \right\vert^{\frac{n}{2}}+\vert \mathcal{L}_\xi \check{Z}_\alpha(x)\vert \right)\to\infty.
\end{equation}
\par\textbf{A cluster with only the first type of points, \textit{i.e.} where all bubbles are of a comparable size.} Assume $\check{x}_{i,\alpha}$ corresponds to the first type. Since for all $j\leq M_{R,\alpha}$,
\begin{equation}
\vert \check{x}_{i,\alpha}-\check{x}_{j,\alpha}\vert^{\frac{n-2}{2}}\check{v}_\alpha(\check{x}_{i,\alpha})\geq 1,
\end{equation}
then
\begin{equation}
\check{v}({\check{x}_{i,\alpha}})\geq 2C(R).
\end{equation}
Since $\check{v}_\alpha$ is uniformly bounded in $\mathcal{C}^2$, there exists $r_i>0$ such that
\begin{equation}
\inf_{B_{\check{x}_{i,\alpha}}(r_i)}\check{v}_\alpha \geq C(R).
\end{equation}
By following the arguments of Lemmas \ref{lem: first estimate} and \ref{first blow up}, there exists a $\mathcal{C}^2(B_0(R))$ limit, \begin{equation}
\check{v}=\lim_{\alpha\to\infty}\check{v}_\alpha
\end{equation}
such that
\begin{equation}
\Delta_\xi \check{v}=f(0)\check{v}^{q-1};
\end{equation}
since $\check{v}$ has at least two maxima, this leads to a contradiction. 
\par \textbf{A cluster with both type of points, \textit{i.e.} where there exists at least one pair of bubbles such that one is much greater than the other.} Around the second type of concentration point, we consider two cases:
either
\begin{equation}
\sup_{B_{\check{x}_{j,\alpha}}\left(8\right)} \check{v}_\alpha(x)\leq M\quad \text{ and }\quad \sup_{B_{\check{x}_{j,\alpha}}\left(8\right)}\left\vert\frac{\nabla \check{v}_\alpha(x)}{\check{v}_\alpha(x)} \right\vert^n+ \left\vert\frac{\nabla^2 \check{v}_\alpha(x)}{\check{v}_\alpha(x)} \right\vert^{\frac{n}{2}}+\vert \mathcal{L}_\xi \check{Z}_\alpha(x)\vert\to\infty
\end{equation}
or 
\begin{equation}
\sup_{B_{\check{x}_{j,\alpha}}\left(8\right)} \check{v}_\alpha(x)\to \infty.
\end{equation}
By similar arguments to those of Lemma \ref{first blow up}. From Lemma \ref{profile on Br}, we know that
\begin{equation}
\vert \check{v}_\alpha-\check{B}_\alpha\vert =o(\delta_\alpha^{\frac{n-2}{2}}).
\end{equation}
where
\begin{equation}
\check{B}_\alpha(x)= \check{\mu}_\alpha^{\frac{n-2}{2}}\left(\check{\mu}_\alpha^{\frac{n-2}{2}}-\frac{\check{f}({\check{x}_{i,\alpha}})}{n(n-2)}\vert x\vert^2\right)
\end{equation}
with
\begin{equation}
\check{\mu}_\alpha=\frac{\mu_\alpha}{\delta_\alpha}=\check{u}_\alpha(\check{x}_{i,\alpha})^{-q+2}.
\end{equation}
Up to a subsequence,
\begin{equation}
\check{u}_\alpha(\check{x}_{j,\alpha})\check{u}_\alpha(x)\to\frac{\lambda_i}{\vert x -\check{x}_j \vert^{n-2}}+H_j(x)
\end{equation}
in $B_{\check{x}_i}\left(\frac{1}{2}\setminus \{\check{x}_j\}\right)$, with $\lambda_j>0$, where $H_j$ is superharmonic in $B_{\check{x}_{j,\alpha}}\left(\frac{1}{2} \right)$ with $H(\check{x}_i)=0$. This means that $\check{u}_\alpha\to 0$ in $\mathcal{C}^0\left(B_{\check{x}_{i,\alpha}}\left(\frac{1}{2}\right)\setminus B_{\check{x}_{i,\alpha}}\left(\frac{1}{4} \right)\right)$ By the Harnack type result, Lemma \ref{harnack lemma}, we get a contradiction.
\par\textbf{A cluster with only the second type of points.}
Let $\check{G}_\alpha(x,\cdot)$ be the Green function of the operator $\Delta_\xi+\delta_\alpha^2 \check{h}_\alpha$ in $B_x(3R)$. It converges to the Green function of $\Delta_\xi$ in $\mathcal{C}^1_{loc}(B_x(3R)\setminus\{x\})$. Since $\Delta_\xi+h_0$ is coercive, for any $y\in B_x(2R)$, and since $Y_\alpha\to 0$,
\begin{equation}
\begin{array}{r l}
\check{u}_\alpha(x) \geq & \int_{B_0\left(\frac{1}{2}\right)}\check{G}_\alpha(x,y)\check{f}_\alpha(y)\check{v}_\alpha^{q-1}(y)\,dy   \\
 & +\int_{B_{\check{x}_{2,\alpha}}\left(\frac{1}{2}\right)}\check{G}_\alpha(x,y)\check{f}_\alpha(y)\check{v}_\alpha^{q-1}(y)\,dy.
\end{array}
\end{equation}
This yields
\begin{equation}
\check{v}_\alpha(x)\geq (1+o(1))(\check{B}_{1,\alpha}(x)+\check{B}_{2,\alpha}(x))-\frac{C}{R^{n-2}}\left(\mu_{1,\alpha}^{\frac{n-2}{2}}+\mu_{2,\alpha}^{\frac{n-2}{2}} \right)
\end{equation}
For $\vert x\vert\leq \frac{1}{4}$, $x\not = 0$, we approximate the RHS with $\check{B}_{1,\alpha}$ to get
\begin{equation}
\left(\frac{\mu_{1,\alpha}}{\mu_{2,\alpha}} \right)^{\frac{n-2}{2}}\vert x\vert ^{n-2}\left(\vert x-\check{x}_2\vert^{2-n}-CR^{2-n}+o(1) \right)\leq o(1)+\frac{C}{R^{n-2}}\vert x\vert^{2-n}
\end{equation}
We divide the previous equation by $\vert x\vert$ and take $x\to 0$ to get, for $R$ large,
\begin{equation}
\limsup_{\alpha\to\infty}\left(\frac{\mu_{1,\alpha}}{\mu_{2,\alpha}} \right)^\frac{n-2}{2}\leq C\frac{16^{n-2}}{R^{n-2}-C16^{n-2}}.
\end{equation}
By switching the roles of $\check{x}_{1,\alpha}$ and $\check{x}_{2,\alpha}$, we obtain
\begin{equation}
\limsup_{\alpha\to\infty}\left(\frac{\mu_{2,\alpha}}{\mu_{1,\alpha}} \right)^\frac{n-2}{2}\leq C\frac{16^{n-2}}{R^{n-2}-C16^{n-2}}.
\end{equation}
This is a contradiction.
\end{proof}

\section{Discussion: Is the drift model a better alternative?}

 We recall that not much is known about far-from-CMC solutions. The classical conformal method seems to display a number of singularities, and these singularities are sometimes difficult to find \textit{a priori} without first solving the corresponding conformal system \cite{Max11,Max15}. 
\par  As we've discussed in the introduction, an advantage of the drift model is that the singularities identified by Maxwell can be found in \textit{a priori} known conformal data sets - \textit{i.e} when the volumetric momentum is null. 
\par Apart from being more natural from a physical and geometrical point of view, another feature of Maxwell's model is that it prescribes more than $10$ parameters. At first glance, it ``over-describes" the initial data. An important idea underlying the works presented in the sequel is the hope to use these four additional parameters to ``tilt'' the coordinate system (the other ten parameters) in the neighbourhood of a singularity. Another way to think about this is that the $10$-dimensional manifold of initial data cannot accurately be covered by only one chart; by changing the additional drift parameters whenever we approach of singularity, we essentially switch to a different chart. In this way, we might prove that the set of solutions to the constraint equations does not possess any real singularity, but only ones due to the choice of coordinates. Naively, one might think of a curve having a vertical tangent which is not well parametrized by its x-axis. The price we pay is that the drift system is analytically much more complicated than the classical one.

\par The goal is to find a viable alternative to the conformal method that gives insight into the structure of the set of solutions of the constraint equations. The drift method proposed by Maxwell provides a promising way forward. The following steps may begin to provide a way forward, in order to achieve this goal:
\begin{description}
\item[a.] \textit{Existence for small data}. Verify that Maxwell's system is reasonable, in the sense that it can be solved even in the case of focusing non linearities. An immediate consequence is that the set of solutions is non-empty. For the non-focusing case, existence is proved in \cite{HolMaxMaz18}, whereas for the focusing case, we cite \cite{Val19}.
\item[b.] \textit{Stability}. Check that, given a perturbation of the coefficients, the set of solutions to the perturbed system is bounded. One might always extract a sequence that converges to a solution of the limiting system. This is the purpose of chapter 3.

\item[c.] \textit{The study of bifurcations}. This is where the extra parameters of Maxwell's method might come into play, by allowing us the freedom to continuously change our mapping as needed. Indeed, as proved by Premoselli \cite{Pre15}, there is no hope that a single choice of $N$ and $V$ lead to a nice smooth parametrisation of the set of solutions. Bifurcations must occur. Even in the defocusing case, such bifurcations can occur, as shown by James Dilts, Michael Holst and David Maxwell \cite{DilHolMax17}. Thus, tilting the coordinates (the parameters) in a neighbourhood of these bifurcations is a way to understand them and the extra parameters give an opportunity to do so. 

\begin{figure}
  \begin{center}
    \begin{tikzpicture}

      \draw [thick, ->] (1,1) -- (1,6);
      \draw [thick, ->] (1,1) -- (7,1) node [below] {10 parameters {\color{purple}(+3)}}; 
    \draw (1,1) .. controls ++(20:0) and ++(-90:0.5) .. (6,2.5) ..  controls ++(90:0.5) and ++(-90:0.5) .. (4,3.5)  .. controls ++(90:0.5) and ++(-90:0.5) .. (6.3,4.3) .. controls ++(90:0.5) and ++(-90:0.5) .. (4.3,5) .. controls ++(90:0.5) and ++(-90:0.2) .. (6.7,5.5);
    \draw [dashed] (6,1) -- (6,6);
    \draw [->, thick, purple] (5.7,1) -- (28:7cm);
    \end{tikzpicture}
    \caption{Initial data manifold, parametrized by the drift method.}
  \end{center}
\end{figure}
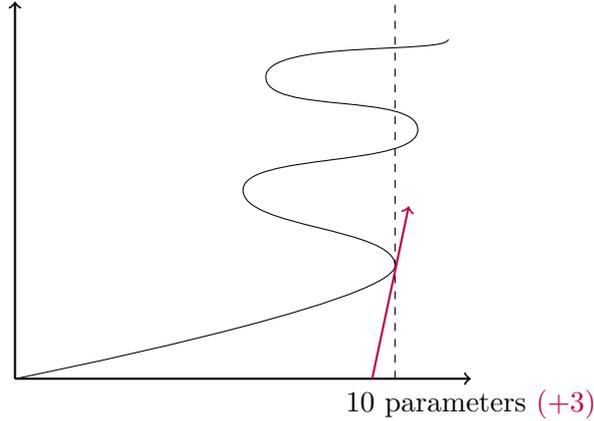

We summarize this program with the help of the following figure. Point a. allows us to start the process of proving that solutions exist for small parameters. Point b. roughly says that the only problem could come from bifurcations corresponding to folding (at least for the parameters for which stability holds). We rule out vertically asymptotic branches. Part c. consists intuitively in tilting the coordinates with the four added parameters, as shown in the figure. These three steps should permit to obtain a nice smooth description of the set of solutions.

\end{description}

\section{Annex}
\subsection{The drift model (continued)}\label{drift method continued}
In order to have a better understanding of the drift method, we recall a basic fact of differential geometry: any metric is uniquely identified by its conformal class together with its volume form. In fact,
\begin{equation}
\mathcal{M}=\mathcal{C}\times\mathcal{V},
\end{equation}
where $\mathcal{M}$ is the space of metrics, $\mathcal{V}$ is the space of volume forms and $\mathcal{C}$ is the space of conformal classes. In the context of the Einstein equations, it makes sense to consider $\mathcal{M}$, $\mathcal{C}$ and $\mathcal{V}$ modulo diffeomorphisms $\mathcal{D}_0$, with $\mathcal{D}_0$ the connected component of the identity in the diffeomorphism group.
\par In his papers, Maxwell describes in great detail how the spaces $\mathcal{M}$, $\mathcal{C}$ and $\mathcal{V}$, together with their tangent, cotangent and quotient spaces, are represented within the choice of parameters \cite{Max14a,Max14b}. The conformal momentum $\mathbf{U}$, for example, is shown to be an element of $T_\mathbf{g}\left(\mathcal{C}\backslash\mathcal{D}_0\right)$. By this interpretation, it becomes clear that $\mathcal{C}$ is prioritized over $\mathcal{V}$ when it comes to choices of parameters.
\par In a 2014 paper, Maxwell introduces a variant to the standard conformal method \cite{Max14b}. Very succinctly, the drift model differs from its predecessor in that it replaces the mean curvature $\tau$ with two new conformal data, a volumetric momentum and a drift. These new quantities are defined by the volumetric equivalent to the York splitting \cite{Max15}:
\begin{equation}
\tau=\tau^*+\frac{1}{N_{\hat{g},\omega}}div (V+Q)
\end{equation}
where $\tau^*\in\R$, $V$ is a smooth vector field and $Q$ is a conformal Killing field.  The volumetric momentum $\tau^*$ as measured by $\omega$ is uniquely determined and can be rewritten as
\begin{equation}\label{volumetric momentum0}
\tau^*=\frac{\int_M N_{\hat{g},\omega}\tau\,dV_{\hat{g}}}{\int_M N_{\hat{g},\omega}\,dV_{\hat{g}}}. 
\end{equation}
The vector field $V$ is uniquely determined up to a $\hat{g}$ divergence-free vector field. 
\par As we explain above, $\tau^*=0$ seems to be a common property of the known non-CMC cases of an infinity of solutions corresponding to the same data set. The drawback of the classical conformal method is that the value of $\tau^*$ cannot be calculated \textit{a priori} from a choice of representatives. One needs to first solve the corresponding system, as 
\begin{equation}
\tau_*=\tau_*(g,u)=\frac{\int_M u^{2q}N_{g,\omega}\tau\,dV_g}{\int_M u^{2q}N_{g,\omega}\,dV_g}.
\end{equation}
Coming back to $(\mathcal{Q}_1)$, this is an argument against the classical conformal model.
\par The volumetric momentum $[g,\tau]_\alpha$ as measured by $\omega$ is $-2\frac{n-1}{n}\tau_*$. A drift $[V]_g^{\text{drift}}$ at $g$ is the equivalence class of $V$, modulo $\text{Ker} \mathcal{L}_g$ and $\text{Ker}\,div_g$. The space of drifts at $g$ is denoted as $\text{Drift}_g$. David Maxwell introduces the concept of drift as an infinitesimal motion in the space of metrics, modulo diffeomorphisms, that preserves conformal class, up to a diffeomorphism, and the volume form, also up to a diffeomorphism.
\par Assumming that $\mathbf{g}$ admits no non-trivial conformal Killing field and therefore that $Q\equiv 0$, one can obtain the initial data $(\hat{g}_{ab},\hat{K}_{ab})$ from a conformal data set, given a gauge $\omega$, as follows.
\begin{enumerate}
\item Choose an arbitrary representative $g_{ab}\in\mathbf{g}$.
\item Choose the unique densitized lapse $N_{g,\omega}$.
\item Choose the unique TT-tensor $U_{ab}$ such that $(g_{ab},U_{ab})=\mathbf{U}$, where $\mathbf{U}$ is the conformal momentum as measured by $\omega$.
\item Choose a vector field $\tilde{V}$, unique up to a conformal Killing field, such that $(g_{ab},\tilde{V}^a)=\mathbf{V}$, where $\mathbf{V}$ is the volumetric drift measured by $\omega$. We use the tilde to differentiate the drift from the potential, while still staying true to Maxwell's initial notation.
\end{enumerate}
Both $u$ and $W$ are unknown. We write
\begin{equation}
\begin{array}{r l}
 \hat{g}_{ab}&=u^{q-2}g_{ab}  \\
 \hat{K}_{ab}&=u^{-2}[\frac{1}{2N_{g,\omega}}(\mathcal{L}_gW)_{ab}+U_{ab}]+\frac{1}{n}u^{q-2}g_{ab}\left(\tau^*+\frac{1}{N_{g,\omega}}div(\tilde{V})\right).
\end{array}
\end{equation}
Plug these quantities into the constraint equations to obtain
\begin{equation}\label{syst of Maxwell0}
\begin{array}{r l}
\Delta_g u+\frac{n-2}{4(n-1)}(R(g)+\vert\nabla\psi\vert_g^2)u=& \frac{(n-2)\vert U+\mathcal{L}_g W\vert^2+\pi^2}{4(n-1)u^{q+1}}\\ 
&+\frac{n-2}{4(n-1)}[2V(\psi)-\frac{n-1}{n}\left(\tau^*+\frac{div_g(u^q\tilde{V})^2}{N_{g,\omega} u^{2q}}\right)]u^{q-1} \\
div_g\left(\frac{1}{2 N_{g,\omega}}\mathcal{L}_g W\right)=&\frac{n-1}{n}u^q d	\left(\frac{div_g(u^{q}\tilde{V})}{2N_{g,\omega}u^{2q}}\right)+\pi\nabla\psi=0.
\end{array}
\end{equation}
The following table regroups for $n=3$ the conformal data and their dimensions (columns 2 and 3), the expressions of physical data as functions of representatives of conformal data (column 1) and the dimensions of the remaining unknowns (column 4).
\renewcommand\arraystretch{1.7}
\begin{equation}
\begin{array}{|*4{>{\displaystyle}c|}}
  \hline%
  \text{Physical data} & \text{Parameters} & \text{Dimensions} & \text{Unknowns}\\ \hline 
  \hat{g}=u^{q-2}g & \mathbf{g} & 5 & 1\\ \hline
   \begin{tabular}{l}
    $\hat{K}_{ab}=u^{-2}[\frac{1}{2 N_{g,\omega}}(\mathcal{L}_gW)_{ab}+U_{ab}]$  \\
    $\quad+\frac{1}{n}u^{q-2}g_{ab}\left(\tau^*+\frac{1}{N_{g,\omega}}div(V)\right)$ 
   \end{tabular}
    & \mathbf{U},\tau^*,\mathbf{N},\mathbf{V} & 2+1+3&3\\[2pt] \hline
  \hat{\psi}=\psi & \psi & 1 & 0 \\ \hline
  \hat{\pi}=u^{-q}\pi & \pi & 1 & 0 \\ \hline
\end{array}
\end{equation}
This time, we obtain additional parameters. More on this in the following section.
\subsection{Standard elliptic theory for the Lamé operator}
If $X$ is a $1$-form in $M$, the Lamé operator is written in coordinate form as:
\begin{equation}
\overrightarrow{\Delta_g}X_i=\nabla^j\nabla_j X_i+\nabla^j\nabla_iX^j-\frac{2}{n}\nabla_i\left(div_g X\right).
\end{equation}
The operator $\overrightarrow{\Delta}_g$ is uniformly elliptic on $M$. It satisfies the strong ellipticity condition (also known as the Legendre-Hadamard condition): for any $x\in M$ and any $\eta\in T_x^* M$ :
\begin{equation}
(L(x,\xi)\eta)_i\eta^i=\vert\xi\vert^2_g\vert\eta\vert^2_g+\left(1-\frac{2}{n}\right)\vert\langle\xi,\eta\rangle\vert^2_g\geq\vert\xi\vert^2_g\vert\eta\vert_g^2.
\end{equation}
The Lamé operator is self-adjoint on $H^1(M)$ on any closed manifold $M$, since by integration by parts one gets, for any $1$ forms $X$ and $Y$,
\begin{equation}
\int_M\langle\overrightarrow{\Delta}_g X,Y\rangle_g\,dv_g=\frac{1}{2}\int_M\langle\mathcal{L}_gX,\mathcal{L}_g Y\rangle_g\,dv_g.
\end{equation}
This implies that for any $1$-form $X$ on $M$,
\begin{equation}
\overrightarrow{\Delta}_g X=0\iff \mathcal{L}_g X=0.
\end{equation}
The standard elliptic theory for (self-adjoint) strongly elliptic operators acting on vector bundles on a compact manifold apply (see Theorem 5.20 in Giaquinta-Martinazzi):
\begin{prop} For any $p>1$, there exists constants $C_1=C_1(g,p)$ and $C_2=C_2(g,p)$ such that for any $1$-form $X$ in $M$:
\begin{equation}
\vert\vert X\vert\vert_{W^{2,p}(M)}\leq C_1\vert\vert\overrightarrow{\Delta}_g X\vert\vert_{L^p(M)}+C_2\vert\vert X\vert\vert_{L^1(M)}.
\end{equation}
In addition, $X$ satisfies
\begin{equation}
\int_M\langle X,K\rangle_g\,dv_g=0
\end{equation}
for all conformal Killing $1$-forms $K$, then we can choose $C_2=0$.
\end{prop}
We now turn to the case of $\R^n$. For any $1\leq i\leq n$, we define the $1$-form $\R^n\setminus\{0\}$ by:
\begin{equation}
\mathcal{G}_i(y)_j=-\frac{1}{4(n-1)\omega_{n-1}}\vert y\vert^{2-n}\left((3n-2)\delta_{ij}+(n-2)\frac{y_i y_j}{\vert y\vert^2}\right)
\end{equation}
for any $y\not= 0$. Note that the matrices $(\mathcal{G}_i(y)_j)_{ij}$ thus defined are symmetric: for any $y\not=0,$
\begin{equation}
\mathcal{G}_i(y)_j=\mathcal{G}_j(y)_i.
\end{equation}
Let $X$ be a field of $1$-form in $\R^n$. For any $R>0$ and for any $x\in B_0(R)$ there holds:
\begin{equation}
\begin{array}{r l}
 X_i(x)=&\int_{B_0(R)}\mathcal{G}_i(x-y)_j\overrightarrow{\Delta}_\xi X(y)^j\,dx \\&+\int_{\partial B_0(R)}\mathcal{L}_\xi X(y)^{kl}\nu_k(y)\mathcal{G}_i(x-y)_l\,d\sigma  \\
 & -\int_{\partial B_0(R)}\mathcal{L}_\xi\left(\mathcal{G}_i(x-\cdot)\right)_{kl}(y)\nu(y)^k X(y)^l\,d\sigma.
\end{array}
\end{equation}
If $Y$ is a smooth $1$-form in $L^1(\R^n)$, then
\begin{equation}
W_i(x)=\int_{\R^n}\mathcal{G}_i(x-y)_j Y^j(y)\,dy=(\mathcal{G}*Y)_i(x)
\end{equation}
satisfies 
\begin{equation}
\overrightarrow{\Delta}_\xi W_i(x)=Y_i(x).
\end{equation}
The system \eqref{systemshort2} is invariant up to adding a conformal Killing $1$-form in $M$ to $W_\alpha$. Let
\begin{equation}
K_R=\{X\in H^1(M)\left(B_0(R)\right),\mathcal{L}_\xi X=0\}
\end{equation}
is the subspace of $1$-forms associated to the kernel to the Neumann problem for $\Delta_\xi$ in $B_0(R)$. The $H^1$ orthogonal space is defined as the space of $1$-forms $Y\in H^1\left(B_0(R)\right)$ such that for any $X\in K_R$:
\begin{equation}
\int_{B_0(R)}\langle Y,K\rangle_\xi\,dx=0.
\end{equation}
For any $1$-form $X\in B_0\left(B_0(R)\right)$, we define the orthogonal projection on $K_R$ by
\begin{equation}
\pi_R(X)=\sum_{j=1}^m\left(\int_{B_0(R)}\langle K_j, X\rangle\,dx\right)K_j.
\end{equation}
The existence of Green $1$-forms satisfying Neumann boundary conditions:
\begin{prop}
For any $1\leq i\leq n$ and any $R>0$, there exists a unique $\mathcal{G}_{i,R}$ defined in $B_0(R)\times B_0(R)\setminus D$, where $D=\{(x,x),x\in B_0(R)\}$ there holds:
\begin{equation}
\begin{array}{r l}
 \left(X-\pi_R(X)\right)_i(x)=&\int_{B_0(R)}\mathcal{G}_{i,R}(x,y)_j\overrightarrow{\Delta}_\xi X(y)^j\,dx  \\
 & +\int_{\partial B_0(R)}\mathcal{L}_\xi X(y)^{kl}\nu_k(y)\mathcal{G}_{i,R}(x,y)_k\,d\sigma.
\end{array}
\end{equation}
Moreover, $\mathcal{G}_{i,R}$ is continuously differentiable in $B_0(R)\times B_0(R)\setminus D$. Furthermore, if $K$ denotes any compact set in $B_0(R)$, there holds for any $x,y\in M$
\begin{equation}
\vert x-y\vert \vert\nabla \mathcal{G}_{i,R}(x,y)\vert +\vert\mathcal{G}_{i,R}(x,y) \vert\leq C(\delta)\vert x-y\vert^{2-n},
\end{equation}
where 
\begin{equation}
\delta=\frac{1}{R}d\left(K, \partial B_0(R)\right)>0.
\end{equation}
\end{prop}
\subsection{Limiting equation}
The following lemma has been proved in \cite{Val19}.
\begin{lemma}\label{lem: A12} Let $u$ be a bounded subharmonic function defined on $\R^n$. If there exists $0<\varepsilon\leq u$ which bounds $u$ from below and $\alpha>0$ such that $u^{-\alpha}$ is a subharmonic function, then $u$ is a constant.
\end{lemma}
\begin{proof}
Let us denote
$$\bar{u}_{x}(R):=\frac{1}{\omega_{n-1}R^{n-1}}\int_{\partial B_{x}(R)}u(y)\,dy$$
the average of a smooth function $u$ over the sphere $\partial B_{x}(R)$. We will sometimes use the simplified notation $\bar{u}(R)$. Recall that, given any subharmonic function $u$, $x\in\R^n$ and for any two radii $R\leq\tilde{R}$, then
\begin{equation}\label{decreasing averages}
\bar{u}_{x}(R)\leq \bar{u}_x(\tilde{R}).
\end{equation}
This follows from
$$r^{n-1}\bar{u}'(r)=\frac{1}{\omega_{n-1}}\int_{\partial B_x{(r)}}\partial_\nu u(y)\,dy=-\frac{1}{\omega_{n-1}}\int_{B_x{(r)}}\Delta u(y)\,dy\geq 0$$
where $r>0$ and $\nu$ is the exterior normal.
\par Note that $u^{-\alpha}\leq \varepsilon^{-\alpha}$ implies that the average of $u^{-\alpha}$ on arbitrary subsets is uniformly bounded. Let us fix $x\in\R^n$. Since $u^{-\alpha}$ is bounded, there exists a constant $M>0$ and a sequence of radii $R_i\to\infty$ as $i\to\infty$ such that
\begin{equation}\label{def M}
M^{-\alpha}:=\lim_{i\to\infty}\overline{u^{-\alpha}}_{x}(R_i).
\end{equation}
In fact, because the averages are increasing (\ref{decreasing averages}), any sequence $R\to\infty$ around any point in $\R^n$ leads to the same limit $M$, since one may always find a subsequence of $R_i$ such that $B_x(R_i)$ includes the new sequence.  
\par As $u^{-\alpha}$ is subharmonic,
$$u^{-\alpha}(x)\leq \overline{u^{-\alpha}}_x(R)$$
and therefore $u^{-\alpha}(x)\leq M^{-\alpha}$, or equivalently 
\begin{equation}\label{M smaller than u x}
M\leq u(x).
\end{equation}
\par For $z\in \R^n$, let $R:=|z-x|$ and $\tilde{R}>R$. By Green's representation theorem, we get
\begin{equation}\label{est: from the annex}
\begin{array}{r l}
\displaystyle u(z) &\displaystyle \leq \int_{\partial B_x(\tilde{R})}u(y)\frac{\tilde{R}^2-R^2}{\omega_{n-1}\tilde{R}|z-y|^n}\, dy\\ \\
&\displaystyle\leq \frac{(\tilde{R}+R)\tilde{R}^{n-2}}{(\tilde{R}-R)^{n-1}}\overline{u}_x(\tilde{R}).
\end{array}
\end{equation}
For $\delta>0,$ we denote
$$\Omega_{\delta,R}:=\{z\in\partial B_x(R), u(z)\geq M+\delta\}$$
a subset of $\partial B_x(R)$ and let
$$\theta_{\delta, R}:=\frac{|\Omega_{\delta, R}|}{|\partial B_x(R)|}\in[0,1]$$
be the corresponding relative size of its volume. Note that $\theta_{\delta,R}\to 0$ as $R\to\infty$. Otherwise, if there exists $\varepsilon\in(0,1]$ such that
$$\limsup_{R\to\infty}\frac{|\{z\in\partial B_x(R), u(z)\geq M+\delta\}|}{|\partial B_x(R)|}=\varepsilon$$
then
$$\limsup_{R\to\infty}\overline{u^{-\alpha}}_x(R) \leq \varepsilon(M+\delta)^{-\alpha}+(1-\varepsilon)M^{-\alpha}<M^{-\alpha}$$
which contradicts our definition (\ref{def M}) of $M$.
\par By choosing $R$ large, $\theta_{\delta,R}\leq \delta$. Let
$$\lambda_{\delta,i}:=\bar{u}_x(2^iR)$$
Note that, by (\ref{est: from the annex}), $\lambda_{\delta,i}\leq 3\times 2^{n-2}\lambda_{\delta, i+1}.$ Since
$$u(x)\leq\lambda_{\delta,i}\leq(M+\delta)(1-\theta_{\delta,2^iR})+\lambda_{\delta,i+1}\times\theta_{\delta,2^iR}$$
then, by induction,
$$u(x)\leq (M+\delta)\frac{1-\delta^l}{1-\delta}+\lambda_l\delta^l$$
for all $l\in \N$. As we take $l\to\infty$,
$$u(x)\leq(M+\delta)\frac{1}{1-\delta}$$
for any $\delta>0$, and therefore $u(x)\leq M.$ By (\ref{M smaller than u x}), $u(x)\equiv M.$
\par We may apply the same argument to any other $\tilde{x}\in\R^n$ and obtain the same value $u(\tilde{x})=M$. Indeed, assuming that 
$$\tilde{M}^{-\alpha}:=\lim_{\tilde{R}\to\infty}\overline{u^{-\alpha}}_{\tilde{x}}(\tilde{R})$$
so that $\tilde{M}^{-\alpha}\geq M^{-\alpha}$, then for $\tilde{R}$ large, $\overline{u^{-\alpha}}_{\tilde{x}}(\tilde{R})\geq M^{-\alpha}$. But, at the same time, given any fixed $\tilde{R}$, then for $R$ sufficiently large, by (\ref{est: from the annex}), $\overline{u^{-\alpha}}_{\tilde{x}}(\tilde{R})\leq \overline{u^{-\alpha}}_{x}(R)$. Thus we obtain that $u\equiv M$ in $\R^n$.
\end{proof}

\clearpage
\bibliographystyle{alpha} 
\bibliography{stability} 
\end{document}